\documentclass[english]{article}
\usepackage[T1]{fontenc}
\usepackage[numbers]{natbib}
\usepackage{filecontents}
\usepackage[latin9]{inputenc}
\usepackage[letterpaper]{geometry}
\usepackage{pifont}
\usepackage{float}
\usepackage{amsmath}
\usepackage{graphicx}
\usepackage{color}
\usepackage{epstopdf}
\usepackage{adjustbox}
\usepackage{diagbox}
\usepackage{multirow}
\usepackage{array}
\usepackage{mathrsfs}
\usepackage{amsmath, amssymb, amsthm}
\usepackage{subfigure}
\usepackage{tikz}
\usepackage{pgfplots}
\usepackage[hidelinks]{hyperref}
\usepackage{comment}

\graphicspath{{./Figures/}}

\newtheorem{theorem}{Theorem}[section]

\newtheorem{lemma}{Lemma}[section]
\newtheorem*{remark}{Remark}

\usepackage{amsfonts} 
\newcommand\norm[1]{\Vert#1\Vert}

\newcommand\abs[1]{\lvert#1\rvert}

\newcommand\dt{\,\text{d}t}
\newcommand\dx{\,\text{d}x}

\newcommand{\Nv}{N_\mathrm{v}}
\newcommand{\Ne}{N_\mathrm{e}}

\usepackage{color}
\usepackage{adjustbox}
\usepackage{diagbox}
\definecolor{black}{rgb}{0,0,0}

\definecolor{red}{rgb}{1,0,0}

\definecolor{blue}{rgb}{0,0,1}

\usepackage{multirow}

\makeatother

\usepackage{babel}
\usepackage{authblk}

\title{Computational Multiscale Methods for Linear\\ Poroelasticity with High Contrast\thanks{The authors acknowledge support from the  Germany/Hong Kong Joint Research Scheme sponsored by the German Academic Exchange Service (DAAD) under the project 57334719 and the Research Grants Council of Hong Kong with reference number G-CUHK405/16.}}

\author[2]{Shubin Fu}
\author[1]{Robert Altmann}
\author[2]{Eric T. Chung}
\author[1]{Roland Maier}
\author[1]{Daniel Peterseim}
\author[2]{Sai-Mang Pun\thanks{Corresponding Author}}
\affil[1]{Department of Mathematics, University of Augsburg, Germany}
\affil[2]{Department of Mathematics, The Chinese University of Hong Kong, Hong Kong}
 
\begin{document}
\maketitle
\begin{abstract}
In this work, we employ the Constraint Energy Minimizing Generalized Multiscale Finite Element Method (CEM-GMsFEM) to solve the problem of linear heterogeneous poroelasticity with coefficients of high contrast. The proposed method makes use of the idea of energy minimization with suitable constraints in order to generate efficient basis functions for the displacement and the pressure. These basis functions are constructed by solving a class of local auxiliary optimization problems based on eigenfunctions containing local information on the heterogeneity. Techniques of oversampling are adapted to enhance the computational performance. Convergence of first order is shown and illustrated by a number of numerical tests.
\end{abstract}

\section{Introduction}
Modeling and simulating the deformation of porous media saturated by an incompressible viscous fluid is important for a wide range of applications such as reservoir engineering in the field of geomechanics \cite{zoback2010reservoir} as well as environmental safety due to overburden subsidence and compaction \cite{geomechanics}. A reasonable model should couple the flow of the fluid with the behavior of the surrounding solid. Biot \cite{biot1941general} proposed a poroelasticity model that couples a Darcy flow with the linear elastic behavior of the porous medium. This model consists of coupled equations for the pressure and the displacements. The stress equation represents quasi-static elasticity coupled to the pressure gradients as a forcing term. On the other hand, the pressure equation is a Darcy-type parabolic equation with a time-dependent coupling to the volumetric strain.

Standard numerical methods such as the finite element method can be used to solve the poroelastic system in case that the medium is homogeneous~\cite{ern2009posteriori}, i.e., present parameters are constant. 
If the medium is strongly heterogeneous, however, the discretization of the domain needs to be sufficiently fine to obtain accurate results. Such a method is costly in the sense that the dimension of the resulting linear system is huge and therefore not feasible for practical computations. 

To alleviate the computational burden, model reduction 
techniques such as upscaling and multiscale methods can be applied. 
In upscaling methods \cite{durfolsky1991homo,gao2015numerical,mura2016two,PetS16,wu2002analysis}, one typically upscales the properties of the medium based on the theory of homogenization. The resulting systems can be solved on a coarse grid with standard techniques and the dimensions of the corresponding linear systems are much smaller. 
In multiscale methods \cite{ArPeWY07,ch03,efendiev2009multiscale,hw97,jennylt03}, however, one still solves the problem on a coarse grid but with precomputed multiscale basis functions using local information of the medium. For the problem of linear poroelasticity in highly heterogeneous media, this was done, e.g., in ~\cite{gmsfem_poro1}. Therein, a set of multiscale basis functions is constructed under the framework of the \textit{Generalized Multiscale Finite Element Method} (GMsFEM), see \cite{gmsfem_elasticity,egw10}. Another technique to solve multiscale problems is the \textit{Localized Orthogonal Decomposition method} (LOD), cf.~\cite{EngHMP16,HenP16,HenP13,maalqvist2014localization,Pet16}. This method was applied in \cite{maalqvist2017generalized} to linear heterogeneous thermoelasticity and optimal first-order convergence of the fully discretized system based on LOD and an implicit Euler discretization in time was proven. This approach was transferred to the present poroelastic setting in~\cite{lod_poro}. 

In the present work, we combine the ideas from LOD and GMsFEM as recently proposed in \cite{chung2017constraint}. For this, the basis functions are constructed by the principle of constraint energy minimization. Based on the GMsFEM, we first create so-called snapshot spaces and afterwards perform model reduction within those spaces by locally solving a class of well-designed spectral problems. The convergence of this \textit{Constraint Energy Minimizing GMsFEM} (CEM-GMsFEM) is analyzed in \cite{chung2017constraint}. Therein, it is shown that this method has a convergence rate proportional to the coarse grid size, which remains valid even in the presence of high contrast provided that sufficiently many basis functions are selected. The approach makes use of the ideas of localization \cite{KorPY18,KorY16,maalqvist2014localization,Pet16} and oversampling to compute multiscale basis functions in some oversampled subregions with the aim to obtain an appropriate orthogonality condition. 

Adopting the idea of CEM-GMsFEM, we propose a multiscale method for the problem of linear poroelasticity with high contrast and construct multiscale spaces for both, the pressure and the displacement. Based on the previous work \cite{lod_poro}, we prove the first-order convergence of the implicit Euler scheme combined with CEM-GMsFEM for the spatial discretization. Numerical results are provided to demonstrate the efficiency of the proposed method. While in the LOD approach the number of basis functions is limited by the number of coarse nodes, the present CEM-GMsFEM setting allows to add additional basis functions in a flexible way based on spectral properties of the differential operators. This improves the accuracy of the method in the presence of high contrast. It is shown that if enough basis functions are selected, the convergence of the method can be shown independently of the contrast. Unfortunately, a high number of basis functions directly influences the computational complexity of the method. The direct influence of the contrast on the needed number of basis functions is not known but numerical results indicate that a moderate number of basis functions, depending logarithmically on the contrast, seems sufficient. 

The paper is organized as follows. In Section \ref{sec:modelpb}, we introduce the model problem and the functional analytical setting. The framework of CEM-GMsFEM including the construction of the basis functions and the resulting fully discrete method are presented in Section \ref{sec:offline}. In Section \ref{sec:convergence}, we analyze the method and provide the corresponding convergence results. Numerical experiments, proving the expected performance of the proposed method, are shown in Section~\ref{sec:num}.

\section{Preliminaries}\label{sec:modelpb}
\subsection{Model problem}
Let $\Omega \subset \mathbb{R}^d$ ($d = 2, 3$) be a bounded and polyhedral Lipschitz domain and  $T > 0$ a fixed time. We consider the problem of linear poroelasticity where we are interested in finding the pressure 
$p\colon [0,T]\times \Omega\to \mathbb{R}$ and the displacement field $u\colon [0,T]\times \Omega\to \mathbb{R}^d$ satisfying
\begin{subequations}
\label{eq:model}
\begin{alignat}{3}
\label{eq:model1}
-\nabla\cdot \sigma(u) +  \nabla (\alpha p)  &= 0\phantom{f}\qquad\text{in } (0,T] \times \Omega, \\
\label{eq:model2}
\partial_t \bigg( \alpha \nabla\cdot  u + \frac{1}{M}  p \bigg)  - \nabla \cdot \bigg( \frac{\kappa}{\nu} \nabla p\bigg) &= f\phantom{0}\qquad\text{in } (0,T] \times \Omega
\end{alignat}
\end{subequations}	
with boundary and initial conditions
\begin{subequations}
\begin{alignat}{3}
\label{eq:init1}
u&=0\phantom{p^0}\qquad \text{on }(0,T]\times \partial\Omega,\\
\label{eq:init2}
p&=0\phantom{p^0}\qquad \text{on }(0,T]\times \partial\Omega,\\
p(\cdot,0)&=p^0\phantom{0}\qquad\text{in }\Omega.
\end{alignat}
\end{subequations}
For the sake of simplicity, we only consider homogeneous Dirichlet boundary here. The extension to other types of boundary conditions is straightforward. In this model, the primary sources of the heterogeneity are the stress tensor $\sigma$, the permeability $\kappa$, and the Biot-Willis fluid-solid coupling coefficient $\alpha$. We denote by $M$ the Biot modulus and by $\nu$ the fluid viscosity. Both are assumed to be constant. Moreover, $f$ is a source term representing injection or production processes. Body forces, such as gravity, are neglected. In the case of a linear elastic stress-strain constitutive relation, the stress and strain tensors may be expressed as
 \begin{equation*}
 {\sigma(u)} = 2\mu  {\epsilon(u)} + \lambda (\nabla\cdot  {u}) \,  {\mathcal{I}} \quad \text{and} \quad
 {\epsilon}( {u}) = \frac{1}{2} \Big( \nabla  {u} + (\nabla  {u})^T \Big),
 \end{equation*}
 where $\mathcal{I}$ is the identity tensor and $\lambda,\, \mu>0$ are the Lam\'e coefficients,
 which can also be expressed in terms of the Young's modulus $E>0$
 and the Poisson ratio $\nu_p\in(-1, 1/2)$,
 \begin{equation*}
 \lambda=\frac{\nu_p}{(1-2\nu_p)(1+\nu_p)}E, \quad
 \mu=\frac{1}{2(1+\nu_p)}E.
 \end{equation*}
In the considered case of heterogeneous media, the coefficients $\mu$, $\lambda$, $\kappa$, and $\alpha$ may be highly oscillatory.
 
 \subsection{Function spaces}
In this subsection, we clarify the notation used throughout the article. 
We write $(\cdot,\cdot)$ to denote the inner product in $L^2(\Omega)$ and $\norm{\cdot}$ for the corresponding norm. 
Let $H^1(\Omega)$ be the classical Sobolev space with norm $\norm{v}_1 := \big( \norm{v}^2 + \norm{\nabla v}^2 \big)^{1/2}$ and $H_0^1(\Omega)$ the subspace of functions having a vanishing trace. We denote the corresponding dual space by $H^{-1}(\Omega)$. 
Moreover, we write $L^r(0,T; X)$ for the Bochner space with the norm 
$$ \norm{v}_{L^r(0,T;X)} := \bigg( \int_0^T \norm{v}_X^r \dt \bigg)^{1/r}, \quad 1\leq r < \infty,$$
$$ \norm{v}_{L^\infty(0,T;X)} := \sup_{0 \leq t \leq T} \norm{v}_X,$$
where $(X,\norm{\cdot}_X)$ is a Banach space. Also, we define $H^1(0,T;X) := \{ v \in L^2(0,T;X) : \partial_t v \in L^2(0,T;X) \}$.
To shorten notation, we define the spaces for the displacement $u$ and the pressure $p$ by
  \begin{equation*}
  V_0:=[H_0^1(\Omega)]^d,\quad Q_0:=H^1_0(\Omega).
  \end{equation*}
\subsection{Variational formulation and discretization}
In this subsection, we provide the variational formulation corresponding to the system \eqref{eq:model}. We first multiply the equations \eqref{eq:model1} and \eqref{eq:model2} with test functions from $V_0$ and $Q_0$, respectively. Then, applying Green's formula and making use of the boundary conditions \eqref{eq:init1} and \eqref{eq:init2}, we obtain the following variational problem: find $u(\cdot,t)\in V_0$ and $p(\cdot,t)\in Q_0$ such that
\begin{subequations}\label{eq:weak}
	\begin{alignat}{3}
   a(u,v) - d(v,p) &= 0, \label{eqn:v1} \\
   d(\partial_t u,q) + c(\partial_t p,q) + b(p,q) &= (f,q) \label{eqn:v2}      
\end{alignat} 
for all $v\in V_0$, $q\in Q_0$ and  
\begin{alignat}{3}
p(\cdot,0)=p^0 \in Q_0. \label{eqn:v3}
\end{alignat}
\end{subequations}
The bilinear forms are defined by 
   \begin{eqnarray*}
   && a(u,v) = \int_{\Omega} \sigma(u) : \epsilon(v)\dx, \qquad\quad b(p,q) = \int_{\Omega} \frac{\kappa}{\nu}\, \nabla p\cdot \nabla q \dx, \\
   && c(p,q) = \int_{\Omega} \frac{1}{M}\, p\, q\dx, \qquad\qquad\quad\,\, d(u,q) = \int_{\Omega} \alpha\,  (\nabla \cdot u)q\dx.
   \end{eqnarray*}   
Note that \eqref{eqn:v1} can be used to define a consistent initial value $u^0:= u(\cdot,0) \in V_0$. 
Using Korn's inequality \cite{ciarlet1988mathematical}, we get
$$ c_\sigma \norm{v}_1^2 \leq a(v,v) =: \norm{v}^2_a \leq C_\sigma \norm{v}_1^2$$
for all $v \in V_0$, where $c_\sigma$ and $C_\sigma$ are positive constants. Similarly, there exist two positive constants $c_\kappa$ and $C_\kappa$ such that 
$$ c_\kappa \norm{q}_1^2 \leq b(q,q) =: \norm{q}^2_b \leq C_\kappa \norm{q}_1^2$$ 
for all $q \in Q_0$.
The proof of existence and uniqueness of solutions $u$ and $p$ to \eqref{eq:weak} can be found in \cite{showalter2000diffusion}. 

To discretize the variational problem \eqref{eq:weak}, let $\mathcal{T}^h$ be a conforming partition for the computational domain 
$\Omega$ with (local) grid sizes $h_{K}:=\text{diam}(K)$ for $K\in \mathcal{T}^h$ and $h:=\text{max}_{K\in \mathcal{T}^h}h_K$. We remark that $\mathcal{T}^h$ is referred to as the \textit{fine grid}. Next, let $V_h$ and $Q_h$ be the standard finite element spaces of first order with respect to the fine grid $\mathcal{T}^h$, i.e.,
$$ V_h := \{ v \in V_0: v\lvert_K \text{ is a polynomial of degree } \leq 1 \text{ for all } K \in \mathcal{T}^h \}, $$
$$ Q_h := \{ q \in Q_0: q\lvert_K \text{ is a polynomial of degree } \leq 1 \text{ for all } K \in \mathcal{T}^h \}. $$
For the time discretization, let $\tau$ be a uniform time step and define $t_n=n\tau$ for $n =  0,1,\dots,N$ and $T=N\tau$. We use the backward Euler method, i.e., for $n = 1,\dots,N$ and given $p_h^0,\ u_h^0$, we aim to find $u_h^n\in V_h$ and $p_h^n\in Q_h$ such that
\begin{subequations}\label{eq:weak1}
	\begin{alignat}{3}
	a(u_h^n,v) - d(v,p_h^n) &= 0, \label{eq:weak1_a}\\
	d(D_{\tau}u_h^n,q) + c(D_{\tau}p_h^n,q) + b(p_h^n,q) &= (f^n,q)       \label{eq:weak1_b}
	\end{alignat} 
\end{subequations}
for all $v\in V_h$ and $q\in Q_h$. Here,
$D_{\tau}$ denotes the discrete time derivative, i.e., $D_{\tau}u_h^n:=(u_h^n-u_h^{n-1})/\tau$, and $f^n:=f(t_n)$.
The initial value $p_h^0\in Q_h$ is set to be the $L^2$
projection of $p^0\in Q_0$. The initial value $u_h^0$ for the displacement can be obtained by solving 
\begin{equation} \label{eq:initial_u}
a(u_h^0,v)=d(v,p_h^0)
\end{equation}
for all $v \in V_h$.
We remark that this classical approach will serve as a reference solution. The aim of this research is to construct a reduced system based on \eqref{eq:weak1}. To this end, we introduce finite-dimensional multiscale spaces $V_{\text{ms}} \subseteq V_0$ and $Q_{\text{ms}} \subseteq Q_0$, whose dimensions are much smaller, for approximating the solution on some feasible coarse grid.

\section{Construction of the multiscale spaces}\label{sec:offline}
In this section, we construct multiscale spaces on a coarse grid. 
\begin{figure}
	\centering
	\includegraphics[width=3in]{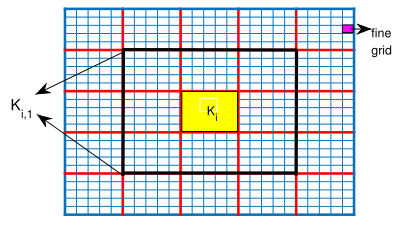}
	\caption{Illustration of the coarse grid $\mathcal{T}^H$, the fine grid $\mathcal{T}^h$, and the oversampling domain $K_{i,1}$.}
	\label{fig:grid}
\end{figure}
Let $\mathcal{T}^H$ be a conforming partition of the computational domain $\Omega$ such that $\mathcal{T}^h$ is a refinement of $\mathcal{T}^H$.
We call $\mathcal{T}^H$ the \textit{coarse grid} and each element of $\mathcal{T}^H$ is a coarse block. We denote with $H:=\text{max}_{K\in \mathcal{T}^H}\text{diam}(K)$ the coarse grid size. 
Let $\Nv$ be the total number of (interior) vertices of $\mathcal{T}^H$ and $\Ne$ be the total number of coarse elements. Let $\{x_i\}_{i=1}^{\Nv}$ be the set of nodes in $\mathcal{T}^H$. Figure \ref{fig:grid} illustrates the
fine grid and a coarse element $K_i$, and the oversampling domain $K_{i,1}$.
The construction of the multiscale spaces consists of two steps. The first step is to construct auxiliary multiscale spaces using the concept of GMsFEM. Based on the auxiliary spaces, we can then construct multiscale spaces containing basis functions whose energy are minimized in some subregions of the domain. These energy-minimized basis functions can then be used to construct a multiscale solution.

\subsection{Auxiliary spaces}
Here, we construct auxiliary multiscale basis functions by solving spectral problems on each coarse element $K_i$ using the spaces $V(K_i):= {V_0 \vert}_{ K_i}$ and $Q(K_i):= {Q_0 \vert}_{ K_i}$. More precisely, we consider the local eigenvalue problems: find $(\lambda_j^i,v_j^i)\in \mathbb{R}\times V(K_i)$ such that
\begin{equation}\label{eq:eig1}
a_i(v_j^i,v)=\lambda_j^i s^1_i(v_j^i,v)
\end{equation}
for all $v \in V(K_i)$ and find 
$(\zeta_j^i,q_j^i)\in \mathbb{R}\times Q(K_i)$ such that 
\begin{equation}\label{eq:eig2}
b_i(q_j^i,q)=\zeta_j^i s^2_i(q_j^i,q)
\end{equation}
for all $q \in Q(K_i)$, where $a_i(u,v) := \int_{K_i} \sigma(u) : \epsilon(v)\dx$,  $b_i(p,q) := \int_{K_i} \frac{\kappa}{\nu} \nabla p\cdot \nabla q \dx$, $s^1_i(u,v):=\int_{K_i} \tilde{\sigma}u\cdot v\dx$, and $s^2_i(p,q):=\int_{K_i} \tilde{\kappa}^2pq\dx$ with 
\begin{equation*}
\tilde\sigma := \sum_{i=1}^{\Nv}(\lambda+2\mu) | \nabla \chi_i^1 |^2,\qquad
\tilde\kappa := \sum_{i=1}^{\Nv}\frac{\kappa}{\nu} | \nabla \chi_i^2 |^2.
\end{equation*}
The functions ${\chi_i^1}$ and ${\chi_i^2}$ are neighborhood-wise defined
partition of unity functions \cite{bm97} on the coarse grid. To be more precise, for $k  =1,2$ the function $\chi_i^k$ satisfies $H \abs{\nabla \chi_i^k} = O(1)$, $0 \leq \chi_i^k \leq 1$, and $\sum_{i=1}^{\Nv} \chi_i^k = 1$. 
Assume that the eigenvalues $\{\lambda_j^i\}$ (resp. $\{ \zeta_j^i \}$) are arranged in ascending order and that the eigenfunctions satisfy the normalization condition $s_i^1(v_j^i,v_j^i)=1$ as well as $s_i^2(q_j^i,q_j^i)=1$. 
Next, choose  $J_i^1\in \mathbb{N}^+$ and define the local auxiliary space $V_{\text{aux}}(K_i):= \text{span} \{v_j^{i}:1\leq j \leq J_i^1 \}$. 
Similarly, we choose $J_i^2 \in \mathbb{N}^+$ and define $Q_{\text{aux}}(K_i) := \text{span} \{q_j^i: 1 \leq j \leq J_i^2\}$. 
Based on these local spaces, we define the global auxiliary spaces $V_{\text{aux}}$ and $Q_{\text{aux}}$ by
$$ V_{\text{aux}} := \bigoplus_{i=1}^{\Ne} V_{\text{aux}}(K_i)\subseteq V_0 \quad \text{and} \quad Q_{\text{aux}} := \bigoplus_{i=1}^{\Ne} Q_{\text{aux}}(K_i)\subseteq Q_0.$$
The corresponding inner products for the global auxiliary multiscale spaces are defined by
\begin{eqnarray*}
s^1(u,v):=\sum_{i=1}^{\Ne}s_i^1(u,v), \qquad 
s^2(p,q):=\sum_{i=1}^{\Ne}s_i^2(p,q) 
\end{eqnarray*}
for all $u,v \in V_{\text{aux}}$ and $p,q\in Q_{\text{aux}}$. Further, we define projection operators $\pi^1: V_0 \to V_{\text{aux}}$ and $\pi^2: Q_0 \to Q_{\text{aux}}$ such that for all $v \in V_0,\ q \in Q_0$ it holds that
\begin{eqnarray*}
\pi^1(v):=\sum_{i=1}^{\Ne}\sum_{j=1}^{J_i^1}s^1_i(v,v_j^i)v_j^i, \qquad
\pi^2(q):=\sum_{i=1}^{\Ne}\sum_{j=1}^{J_i^2}s^2_i(q,q_j^i)q_j^i.
\end{eqnarray*}

\subsection{Multiscale spaces}
In this subsection, we construct the multiscale spaces for the practical computations. 
For each coarse element $K_i$, we define
the oversampled region $K_{i,m}\subset\Omega$ obtained by enlarging $K_i$ by $m$ layers, i.e.,
$$ K_{i,0} := K_i, \quad K_{i,m} := \bigcup \left\{ K\in \mathcal{T}^H : K \cap \overline{K_{i,m-1}} \neq \emptyset \right\}, \quad m = 1,2,\dots,$$
see Figure \ref{fig:grid} for an illustration of $K_{i,1}$. We define $V_0(K_{i,m}):=[H_0^1(K_{i,m})]^d$ and $Q_0(K_{i,m}):=H^1_0(K_{i,m})$. Then, for each pair of auxiliary functions $v_j^i\in V_{\text{aux}}$ and $q_j^i\in Q_{\text{aux}}$,
we solve the following  minimization problems:
find $\psi_{j,m}^i\in V_0(K_{i,m})$ such that
\begin{equation}\label{eq:mineq1}
\psi_{j,m}^i=\text{argmin}\Big\{a(\psi,\psi)+s^1\big(\pi^1(\psi)-v_j^i,\pi^1(\psi)-v_j^i\big):\,\psi\in V_0(K_{i,m})\Big\}
\end{equation}
and 
find $\phi_{j,m}^i\in Q_0(K_{i,m})$ such that
\begin{equation}\label{eq:mineq2}
\phi_{j,m}^i=\text{argmin}\Big\{b(\phi,\phi)+s^2\big(\pi^2(\phi)-q_j^i,\pi^2(\phi)-q_j^i\big):\,\phi\in Q_0(K_{i,m})\Big\}.
\end{equation}
Note that problem (\ref{eq:mineq1}) is equivalent to the local problem 
\begin{equation*}
a(\psi_{j,m}^i,v)+s^1\big(\pi^1(\psi_{j,m}^i),\pi^1(v)\big)=s^1\big(v_j^i,\pi^1(v)\big)
\end{equation*}
for all $v\in V_0(K_{i,m})$, whereas problem (\ref{eq:mineq2}) is equivalent to
\begin{equation*}
b(\phi_{j,m}^i,q)+s^2\big(\pi^2(\phi_{j,m}^i),\pi^2(v)\big)=s^2\big(q_j^i,\pi^2(q)\big)
\end{equation*}
for all $q\in Q_0(K_{i,m})$.
Finally, for fixed parameters $m$, $J_i^1$, and $J_i^2$, the multiscale spaces $V_{\text{ms}}$ and $Q_{\text{ms}}$ are defined by 
$$ V_{\text{ms}} := \text{span}\{ \psi_{j,m}^i: 1\leq j \leq J_i^1,\ 1\leq i \leq \Ne \} \quad \text{and} \quad Q_{\text{ms}} :=\text{span} \{ \phi_{j,m}^i: 1 \leq j \leq J_i^2,\ 1\leq i \leq \Ne\},$$
see also Figure \ref{fig:basis} for an illustration of such a multiscale basis function. 

 The multiscale basis functions can be interpreted as approximations to global multiscale basis functions $\psi_j^i\in V_0$ and $\phi_j^i\in Q_0$, similarly defined by
\begin{align*}
\psi_j^i &= \text{argmin}\Big\{a(\psi,\psi)+s^1\big(\pi^1(\psi)-v_j^i,\pi^1(\psi)-v_j^i\big): \,\psi\in V_0\Big\}, \\
\phi_j^i &= \text{argmin}\Big\{b(\phi,\phi)+s^2\big(\pi^2(\phi)-q_j^i,\pi^2(\phi)-q_j^i\big): \,\phi\in Q_0\Big\}.
\end{align*}
These basis functions have global support in the domain $\Omega$ but, as shown in \cite{chung2017constraint}, decay exponentially outside some local (oversampled) region. This property plays a vital role in the convergence analysis of CEM-GMsFEM and justifies the use of local basis functions in $V_{\text{ms}}$ and $Q_{\text{ms}}$. 
\begin{figure}[H]
	\centering
	\subfigure[first component of $\psi_{1,m}^i \in V_{\text{ms}}$]{
		\includegraphics[width=2.in]{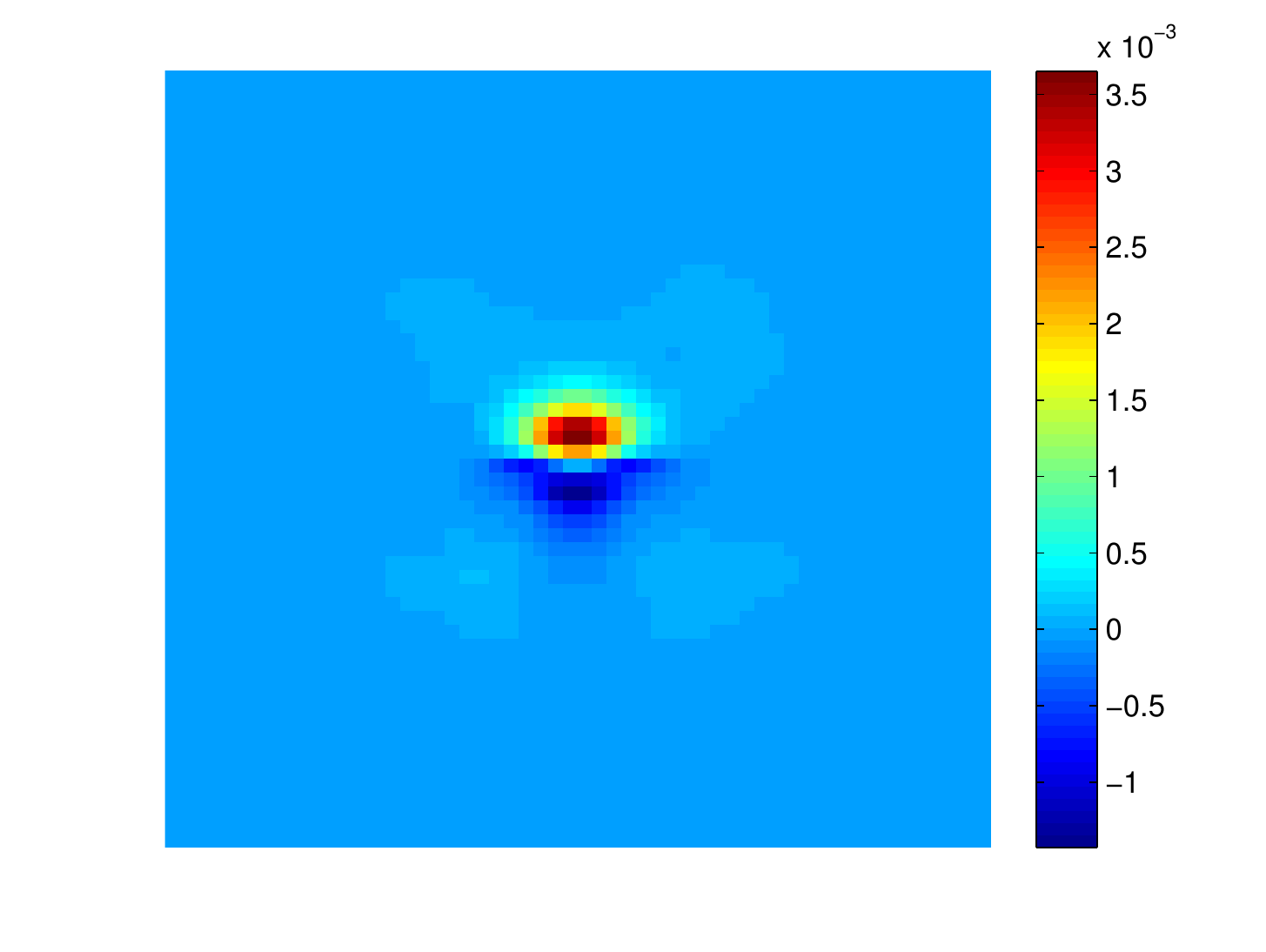}}
	\subfigure[second component of $\psi_{1,m}^i \in V_{\text{ms}}$]{
		\includegraphics[width=2.in]{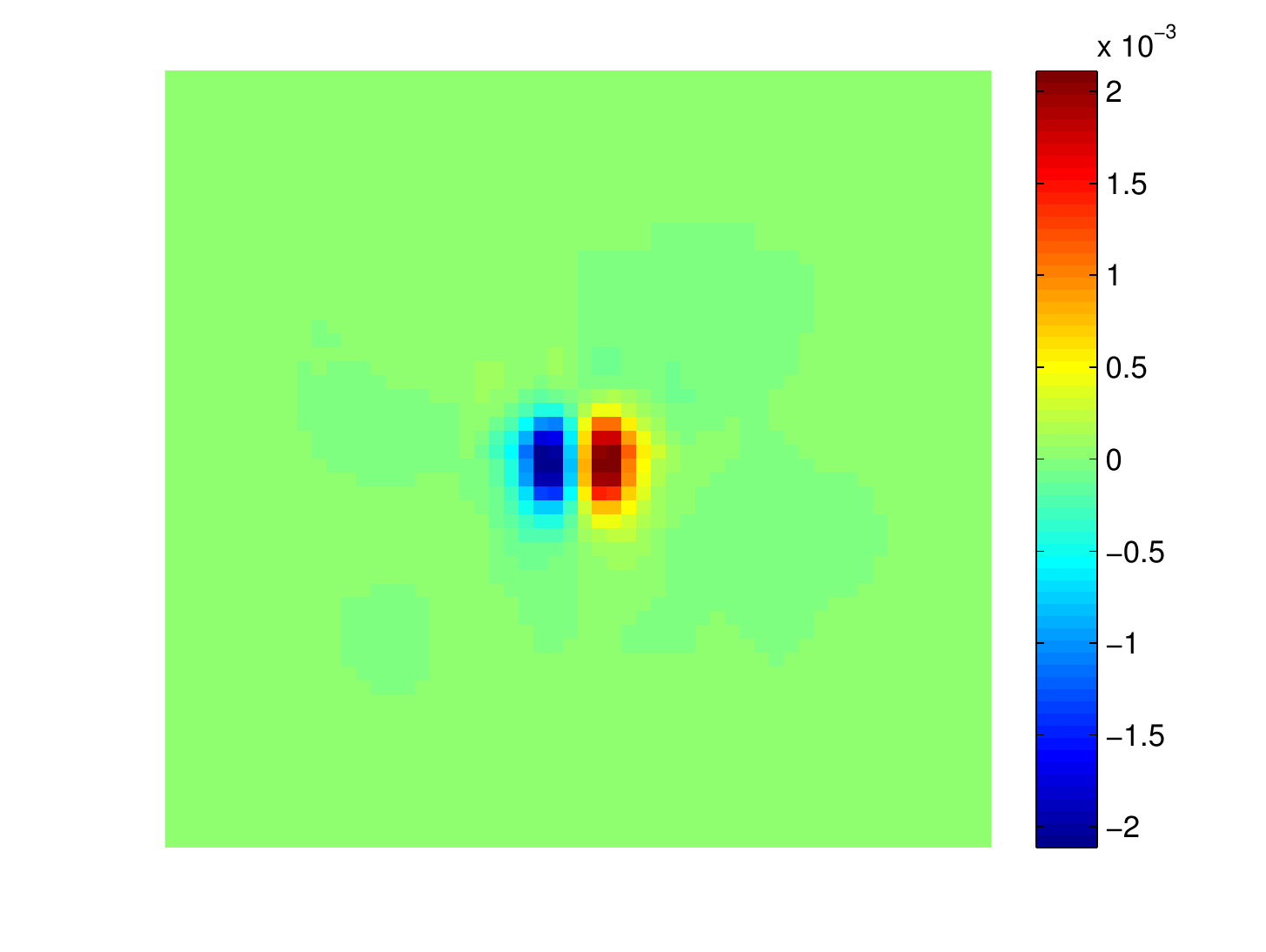}}
	\subfigure[$\phi_{1,m}^i \in Q_{\text{ms}}$]{
		\includegraphics[width=2.in]{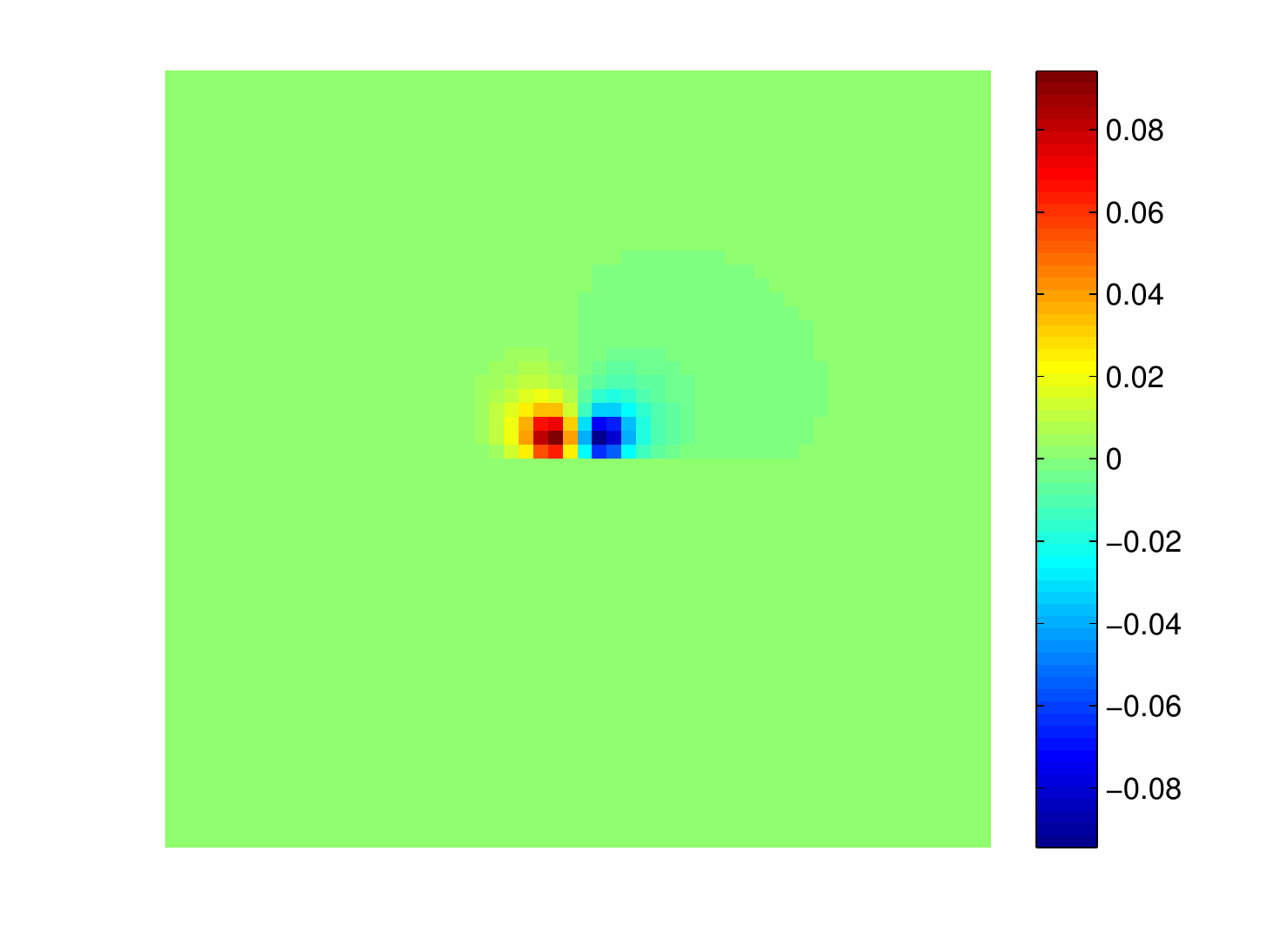}}
	\caption{First multiscale basis function of $V_{\text{ms}}$ and $Q_{\text{ms}}$, respectively, for $m=5$ and $H =\sqrt{2}/40$.}
	\label{fig:basis} 
\end{figure}

\subsection{The multiscale method}
In order to make the multiscale spaces~$V_{\text{ms}}$ and~$Q_{\text{ms}}$ suitable for computations, we need finite-dimensional analogons. For this, we follow the construction of the previous subsections, restricted the finite element space based on the fine grid~$\mathcal{T}^h$. This then yields the following fully discrete scheme: for $n=1,2,\dots,N$ find $(u_{\text{ms}}^n,p_{\text{ms}}^n) \in V_{\text{ms}} \times Q_{\text{ms}}$ that solve
\begin{subequations}\label{eq:weak2}
	\begin{alignat}{2}
	a(u_{\text{ms}}^n,v) - d(v,p_{\text{ms}}^n) &= 0, \\
	d(D_{\tau}u_{\text{ms}}^n,q) + c(D_{\tau}p_{\text{ms}}^n,q) + b(p_{\text{ms}}^n,q) &= (f^n,q)       
	\end{alignat} 
\end{subequations}
for all $(v,q) \in V_{\text{ms}} \times Q_{\text{ms}}$ with initial condition $p_{\text{ms}}^0\in Q_{\text{ms}}$ defined by
$$b(p_h^0 - p_{\text{ms}}^0,q) = 0 $$
for all $q\in Q_{\text{ms}}$. 
  
\section{Convergence analysis}\label{sec:convergence}
In this section, we analyze the proposed multiscale method \eqref{eq:weak2}. First, we recall some theoretical results related to the discretization of the problem of linear poroelasticity with finite elements. Throughout this section, $C$ denotes a generic constant which is independent of spatial discretization parameters and the time step size. Further, the notation $a \lesssim b$ is used equivalently to $a \leq Cb$.

\begin{lemma}[{\cite[Lem. 3.1]{lod_poro}}] \label{lem:well_posed}
Given initial data $p_h^0 \in Q_h$ and $u_h^0 \in V_h$ defined in \eqref{eq:initial_u}, system \eqref{eq:weak1} is well-posed. That is, there exists a unique solution, which can be bounded in terms of the initial values and the source function.
\end{lemma}

\begin{lemma}[{\cite[Thm. 3.3]{maalqvist2017generalized}}] \label{lem:stab}
Assume $f \in L^\infty(0,T; L^2(\Omega)) \cap H^1(0,T;H^{-1}(\Omega))$. Then, for all $n = 1,\dots,N$, the fully discrete solution $(u_h^n,p_h^n)$ of \eqref{eq:weak1} satisfies the stability bound 
$$ \bigg( \tau \sum_{j=1}^n \norm{D_\tau u_h^j}_1^2 \bigg)^{1/2} +\bigg( \tau \sum_{j=1}^n \norm{D_\tau p_h^j}^2 \bigg)^{1/2}  + \norm{p_h^n}_1 \lesssim  \norm{p_h^0}_1 + \norm{f}_{L^2(0,t_n; L^2(\Omega))}.$$
Further, if $p_h^0 = 0$, we have 
$$ \norm{D_\tau u_h^n}_1 + \norm{D_\tau p_h^n} + \bigg( \tau \sum_{j=1}^n \norm{D_\tau p_h^j}_1^2 \bigg)^{1/2} \lesssim \norm{f}_{L^\infty(0,t_n; L^2(\Omega))} + \norm{\partial_t f}_{L^2(0,t_n;H^{-1}(\Omega))}$$
and for $f = 0$ it holds that 
$$ \norm{D_\tau u_h^n}_1 + \norm{D_\tau p_h^n} + t_n^{1/2} \norm{D_\tau p_h^n}_1 \lesssim t_n^{-1/2} \norm{p_h^0}_1.$$
\end{lemma}

\begin{lemma}[{cf. \cite[Thm. 3.1]{ern2009posteriori}}] \label{lem:fine_scale_error}
Assume that the coefficients satisfy $\mu, \lambda, \kappa, \alpha \in W^{1,\infty}(\Omega)$. Further, let the solution $(u,p)$ of \eqref{eq:weak} be sufficiently smooth. Then, for each $n=1,\dots,N$ the fully discrete solution $(u_h^n, p_h^n)$ of \eqref{eq:weak1} satisfies the estimate 
$$ \norm{u(t_n) - u_h^n}_1 + \norm{p(t_n) - p_h^n} + \bigg(\tau \sum_{j=1}^n \norm{p(t_j) - p_h^j}_1^2 \bigg)^{1/2} \leq C_\text{osc} h + C\tau,$$
where the constant $C_\text{osc}$ scales with $\max\{ \norm{\mu}_{W^{1,\infty}(\Omega)},\norm{\lambda}_{W^{1,\infty}(\Omega)},\norm{\kappa}_{W^{1,\infty}(\Omega)},\norm{\alpha}_{W^{1,\infty}(\Omega)}\}$.
\end{lemma}
\begin{remark}
The previous lemma states the first-order convergence of the finite element method but also reveals the dependence of the involved constant on possible oscillations in the coefficients describing the media.
\end{remark}

Next, for any $u \in V_h$, $p\in Q_h$, we define $\widehat{u} \in V_{\text{ms}}$, $\widehat{p}\in Q_{\text{ms}}$ as the elliptic projection of $u$ and $p$ with respect to $a$ and $b$, i.e.,
$$ a(u-\widehat{u},v) = 0 \quad \text{and} \quad b(p-\widehat{p},q) = 0$$
for all $v \in V_{\text{ms}}$ and $q \in Q_{\text{ms}}$.
Moreover, let $\mathcal{A}\colon V_h \to V_h$ and $\mathcal{B}\colon Q_h \to Q_h$ be the Riesz's projections defined by
$$ (\mathcal{A}u, v) := a(u,v), \qquad (\mathcal{B}p,q) := b(p,q)$$ 
for all $v \in V_h$ and $q \in Q_h$.
The elliptic projections $\mathcal A$ and $\mathcal B$ fulfill the following estimates.
\begin{lemma} \label{lem:proj}
Let $V_\mathrm{ms},\,Q_\mathrm{ms}$ be the multiscale spaces defined in Section~\ref{sec:offline} with parameters $m,\ J^1_i,\ J^2_i$ sufficiently large. Then, for all $v\in V_h$ and $q \in Q_h$, it holds that
\begin{eqnarray}
	\norm{v- \widehat{v}} &\lesssim H\, \norm{v - \widehat{v}}_1, \label{eq:ap_bound} \\
	\norm{q- \widehat{q}} &\lesssim H\, \norm{q - \widehat{q}}_1, \label{eq:bp_bound}
\end{eqnarray}
and
\begin{eqnarray}
 \norm{v- \widehat{v}}_1 &\lesssim H\, \norm{\mathcal{A}v},  \label{eq:ap_riesz} \\
 \norm{q-\widehat{q}}_1 &\lesssim H\, \norm{\mathcal{B}q}.  \label{eq:bp_riesz}
\end{eqnarray}
\end{lemma}
\begin{proof} For $v\in V_h$, consider the variational problem
$$ a(z,w) = (v- \widehat{v},w)$$
for all $w \in V_h$. It was shown in \cite[Thm. 2]{chung2017constraint} that for sufficiently large $m$ and $J^1_i$ the ellipticity of $a(\cdot,\cdot)$ implies 
$\norm{z - \widehat{z}}_1 \lesssim H \norm{v - \widehat{v}}_1$.
Therefore, we have 
$$ \norm{v - \widehat{v}}^2 = a(z,v-\widehat{v}) \lesssim \norm{z-\widehat{z}}_1 \norm{v-\widehat{v}}_1 \lesssim H \norm{v-\widehat{v}} \norm{v-\widehat{v}}_1$$
and, thus, $\norm{v- \widehat{v}} \lesssim H \norm{v- \widehat{v}}_1$,
which proves \eqref{eq:ap_bound}. Further, we have
$$ \norm{v- \widehat{v}}_1^2 \lesssim a(v-\widehat{v},v-\widehat{v}) = a(v,v-\widehat{v}) = (\mathcal{A}v,v-\widehat{v}) \leq \norm{\mathcal{A}v} \norm{v-\widehat{v}}.$$
With \eqref{eq:ap_bound} and the inequality above, we get 
$$ \norm{v-\widehat{v}}_1^2 \lesssim H \norm{\mathcal{A}v} \norm{v-\widehat{v}}_1.$$
This proves \eqref{eq:ap_riesz}. Similarly, one may prove \eqref{eq:bp_bound} and \eqref{eq:bp_riesz} using the ellipticity of $b(\cdot,\cdot)$. 
\end{proof}

\begin{remark}
\label{rem_mJ}
In order to achieve the desired estimates in Lemma \ref{lem:proj}, it is necessary to choose the number of oversampling layers $m \approx O\big(\log(HC_p)\big)$, where $C_p$ is a constant depending on the Lam\'{e} coefficients, the permeability~$\kappa$, and $\alpha$, see~\cite[Sect.~6]{chung2017constraint} for more details. Further, the parameters $J^1_i$ and $J^2_i$ account for high contrast within the media and numerical results suggest a logarithmic dependence on the contrast. 
\end{remark}

The main theoretical result reads as follows.
\begin{theorem} \label{thm:main}
Assume sufficiently large parameters $m$, $J^1_i$, $J^2_i$, a source function $f \in L^\infty(0,T; L^2(\Omega)) \cap H^1(0,T;H^{-1}(\Omega))$ as well as initial data $p_h^0 \in Q_h$ and $u_h^0 \in V_h$ defined in \eqref{eq:initial_u}. Then, the error between the multiscale solution $(u_{\mathrm{ms}}^n,p_{\text{ms}}^n) \in V_{\mathrm{ms}} \times Q_{\mathrm{ms}}$ of \eqref{eq:weak2} and the fine-scale solution $(u_h^n,p_h^n) \in V_h \times Q_h$ of \eqref{eq:weak1} satisfies
\begin{eqnarray*}
 \norm{u_h^n - u_{\mathrm{ms}}^n}_1 + \norm{p_h^n - p_{\mathrm{ms}}^n}_1    \lesssim H \mathcal{S}_n+ t_n^{-1/2}H \norm{p_h^0}_1 \label{eq:thm}
 \end{eqnarray*}
for $n = 1,2,\dots,N$. Here, $\mathcal{S}_n$ only depends on the data and is defined by
$$ \mathcal{S}_n :=  \norm{p_h^0}_1  +  \norm{f}_{L^2(0,t_n;L^2(\Omega))}  +  \norm{f}_{L^\infty(0,t_n;L^2(\Omega))} + \norm{\partial_t f}_{L^2(0,t_n;H^{-1}(\Omega))} .$$
\end{theorem}
\begin{proof}
The proof mainly follows the lines of the proof of~\cite[Thm. 3.7]{lod_poro} and makes use of the results of Lemma \ref{lem:proj}. Thus, we only give the ideas of the proof. 

Due to the linearity of the problem, we can decompose the fine-scale solutions into $ u_h^n  = \bar{u}_h^n + \tilde{u}_h^n$ and $p_h^n = \bar{p}_h^n + \tilde{p}_h^n$, where $(\bar{u}_h^n, \bar{p}_h^n)$ solves \eqref{eq:weak1} with $f = 0$ and $(\tilde{u}_h^n, \tilde{p}_h^n)$ solves \eqref{eq:weak1} with $p_h^0 = 0$. With this, the stability estimates in Lemma~\ref{lem:stab} can be applied. 
As a second step, we decompose the difference of the fine-scale and multiscale solution with the help of the introduced projection. The differences $u^n_h-{\widehat{u}}_h^n$ and $p^n_h-{\widehat{p}}_h^n$ can then be estimated by Lemma~\ref{lem:proj}. For the remaining parts ${\widehat{u}}_h^n-u^n_\mathrm{ms}$ and ${\widehat{p}}_h^n - p^n_\mathrm{ms}$ one needs to consider system~\eqref{eq:weak2} with appropriate test functions in combination with properties of the projection and Lemma~\ref{lem:stab}.
\end{proof}
\begin{remark}
Combining Lemma \ref{lem:fine_scale_error} and Theorem \ref{thm:main}, it follows that the proposed multiscale method converges as $H+ \tau$ with respect to the $L^\infty (0,T;V_0)$ norm for the displacement and the $L^2(0,T;Q_0) \cap L^\infty(0,T;L^2(\Omega))$ norm for the pressure.
\end{remark}

\section{Numerical results}\label{sec:num}
In this section, we present some numerical experiments to demonstrate the performance of our method. The computational domain is set to $\Omega:=(0,1)^2$, $T := 100$, and the time step size is chosen as $\tau := 5$. 
We consider homogeneous Dirichlet boundary conditions for both the pressure and the displacement. The initial pressure is $p^0(x):=x_1(1-x_1)x_2(1-x_2)$ and the force $f$ is set to be constant $1$. 
We assume heterogeneous coefficients that have different values in two subdomains. Considering Figure \ref{fig:model}, we call the blue region $\Omega_1$ and the red region $\Omega_2$. 

We choose $\kappa = 1$ in $\Omega_1$ and $\kappa = 10^4$ in $\Omega_2$, which implied high contrast in the parameters. The Young's modulus $E$ is set to be equal to the permeability $\kappa$ and the Biot-Willis fluid-solid coupling coefficient $\alpha$ is chosen uniformly distributed on each element, i.e., $\alpha |_K \sim U[0.5,1]$ for any $K \in \mathcal{T}^H$. We further set $\nu = 1$, $M=1$, and $\nu_p = 0.2$. As shown in Figure~\ref{fig:model}, the first model includes isolated short and long channels, while the second model is a 
fracture-type media including randomly distributed fractures.
The fine grid size is chosen as $h = \sqrt{2}/200$ in both models.

Recall that $m$ is the number of oversampling layers used to compute the multiscale basis and that $H$ denotes the coarse grid size.
Further, we set $J := J_i^1 = J_i^2$ as the number of basis functions used in the auxiliary spaces $V_{\text{aux}}$ and $Q_{\text{aux}}$.
We use the same parameter settings for the construction of $V_{\text{ms}}$ and $Q_{\text{ms}}$.
To quantify the accuracy of CEM-GMsFEM, we define 
relative weighted $L^2$ errors and energy errors for the displacement and the pressure at time $T$,
\begin{equation*}
e^u_{L^2}:=\frac{\norm{(\lambda+2\mu)(u_{\text{ms}}(\cdot,T)-u_h(\cdot,T))}_{L^2(\Omega)}}{\norm{(\lambda+2\mu)u_h(\cdot,T)}_{L^2(\Omega)}},\qquad
e^u_{a}:=\frac{\norm{u_{\text{ms}}(\cdot,T)-u_h(\cdot,T)}_a}{\norm{u_h(\cdot,T)}_a},
\end{equation*}
\begin{equation*}
e^p_{L^2}:=\frac{\norm{\frac{\kappa}{\nu}(p_{\text{ms}}(\cdot,T)-p_h(\cdot,T))}_{L^2(\Omega)}}{\norm{\frac{\kappa}{\nu}p_h(\cdot,T)}_{L^2(\Omega)}},\qquad
e^p_{b}:=\frac{\norm{p_{\text{ms}}(\cdot,T)-p_h(\cdot,T)}_b}{\norm{p_h(\cdot,T)}_b}, 
\end{equation*}
where $(u_h(\cdot,T),p_h(\cdot,T))$ is the reference solution computed on the fine grid and 
$(u_{\text{ms}}(\cdot,T),p_{\text{ms}}(\cdot,T))$ is the multiscale solution obtained by the proposed method \eqref{eq:weak2}. 
Based on the above theory, we expect that the multiscale solution converges linearly in $H$ with respect to the energy norms and that the results may be improved by either using more basis functions (increase $J$) or enlarging the oversampling region (increase $m$).

\begin{figure}
	\centering
	\subfigure[Test model $1$.]{
		\includegraphics[width=2.5in]{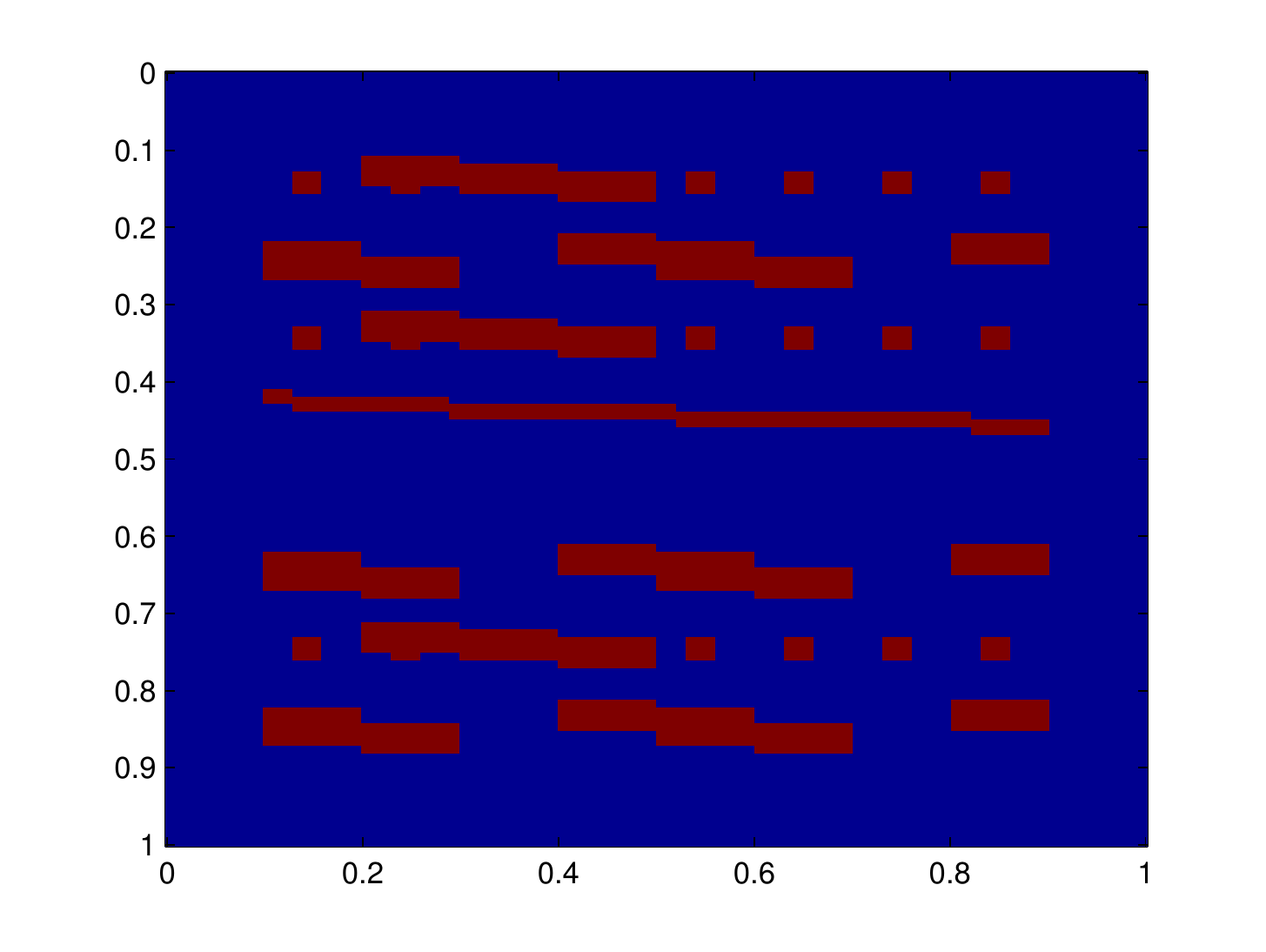}}
	\subfigure[Test model $2$.]{
		\includegraphics[width=2.5in]{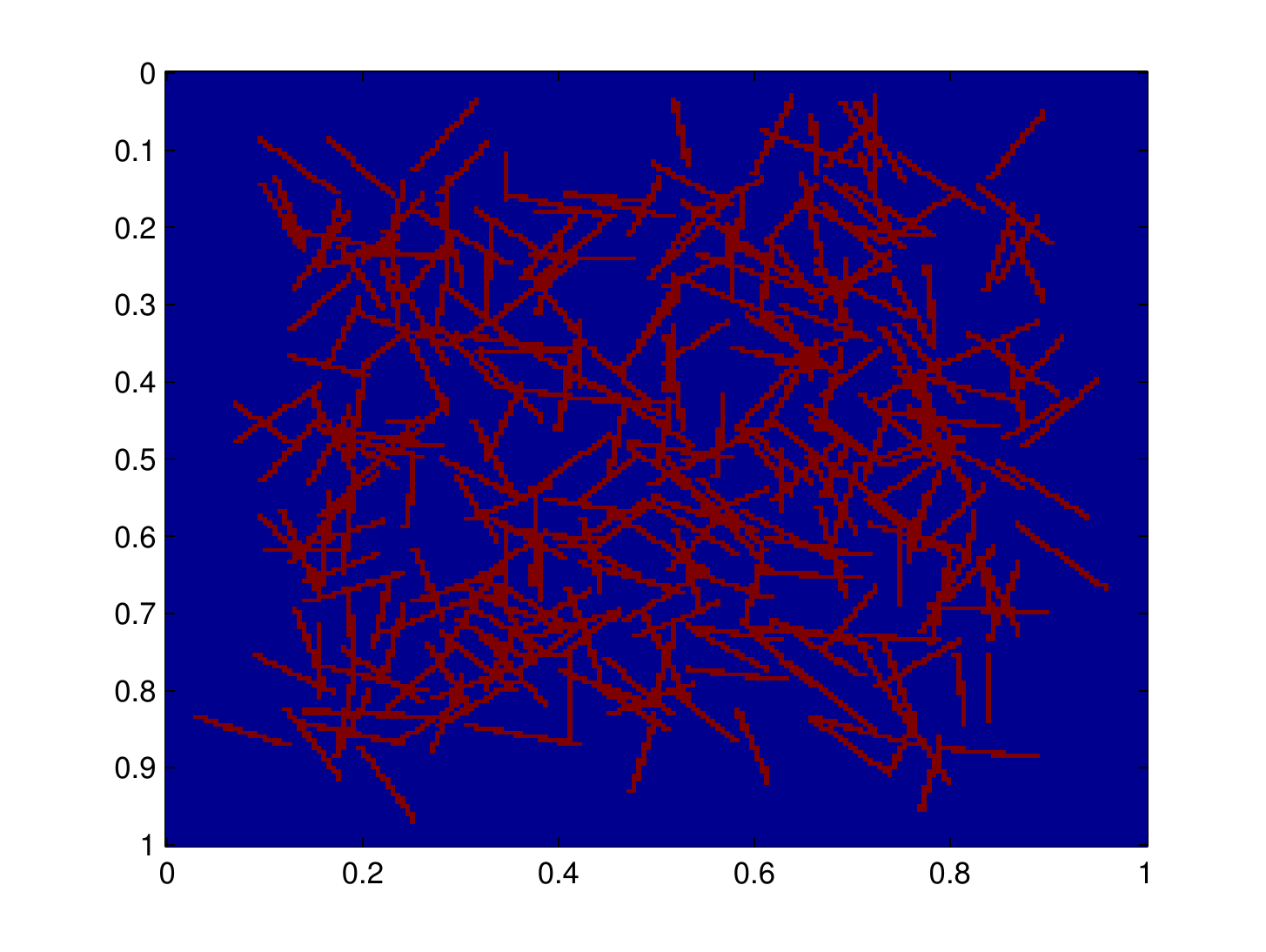}}
	\caption{Illustration of the subdomains $\Omega_1$ (blue) and $\Omega_2$ (red) for the two test models that define the heterogeneous coefficients.}
	\label{fig:model} 
\end{figure}

\subsection{Test model 1}
In this subsection, the numerical results for test model $1$ are presented. 
First, we investigate the convergence behavior of the CEM-GMsFEM solution
with respect to the coarse grid size. As suggested in Remark~\ref{rem_mJ}, we set the number of oversampling layers to $m=4\lfloor\text{log}(H)/\text{log}(\sqrt{2}/10)\rfloor$ and $J=4$ to form the auxiliary spaces. The results are presented in Table \ref{ta:h1}.
One may find that the error decreases as $H$ decreases and that the
CEM-GMsFEM solution converges linearly in the energy norms and quadratically in $L^2$. In Figure \ref{fig:ovnb1} (left), we present the results when using different numbers of oversampling layers with fixed coarse grid size $H = \sqrt{2}/40$.

It is clearly visible that the error decreases faster if more oversampling layers are added. Further, once the number of oversampling layers exceeds a certain number, the error stagnates. To investigate the influence of using different numbers of basis functions, we use a fixed coarse grid with grid size $H=\sqrt{2}/40$ and a fixed number of oversampling layers. The errors in the energy norms are presented in Figure \ref{fig:ovnb1} (right). For $m=6$ there is no observable decay for $e_{b}^p$ by using more basis functions. One of the reasons for that is the fact that the error for the pressure is already very small even if only $1$ basis is used. On the other hand, due to the high contrast, the error of the displacement $e_{a}^u$ decreases as the number of local basis functions increases. 
Once the number of local basis functions exceeds a certain level, the error decays slower. This happens when also the decay of the eigenvalues slows down. Figures \ref{fig:fine1} and \ref{fig:ms1} show the reference solution and the multiscale solution at $T=100$, respectively. One can see that the proposed method captures most of the details of the reference solution.
 
\begin{table}
	\centering \begin{tabular}{c|c|c||c|c|c|c}
		$J$ &$H$& $m$  & $e_{L^2}^u$   & $e_{a}^u$& $e_{L^2}^p$   & $e_{b}^p$ \tabularnewline\hline
		4&$\sqrt{2}$/10	&4&9.41e-03 &  1.14e-01 & 6.05e-03 & 5.79e-02     \tabularnewline\hline
		4&$\sqrt{2}$/20	&5&1.22e-03 &  7.39e-02 & 8.75e-04 & 2.29e-02     \tabularnewline\hline
		4&$\sqrt{2}$/40	&6&2.08e-04 &  2.08e-02 & 1.58e-04 & 9.64e-03    \tabularnewline
	\end{tabular}
	\caption{Numerical results with varying coarse grid size $H$ for test model $1$.} 
\label{ta:h1}
\end{table}

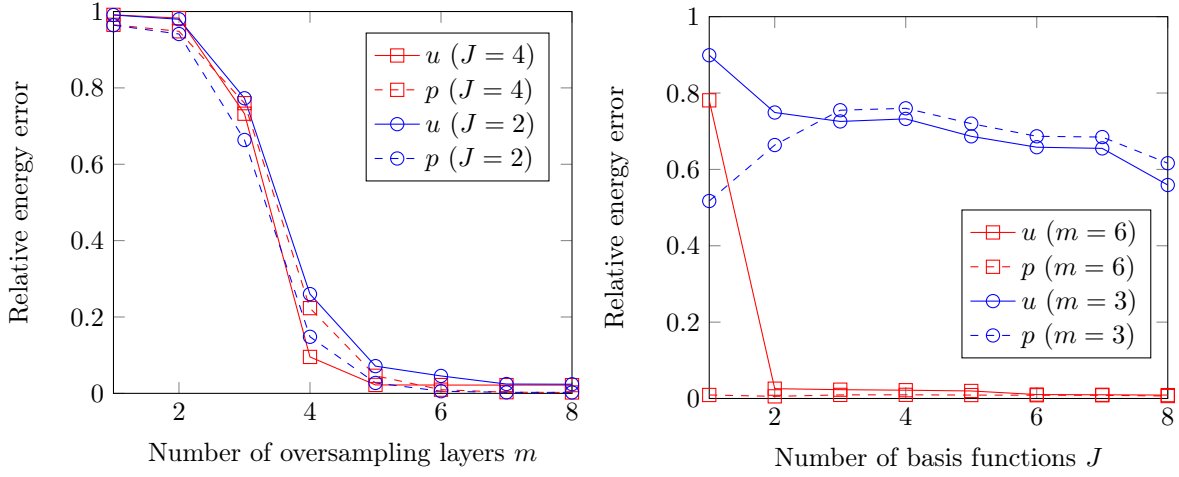
\begin{figure}
\centering
%
%
\begin{tikzpicture}

\begin{axis}[%
width=2.4in,
height=2in,
scale only axis,
xmin=1,
xmax=8,
xlabel={Number of oversampling layers $m$},
ymin=0,
ymax=1,
ylabel={Relative energy error},
axis background/.style={fill=white},
legend style={legend cell align=left, align=left, draw=black},
legend style={at={(0.55,0.75)},anchor=west}
]
\addplot [color=red, mark size=2.5pt, mark=square, mark options=solid]
  table[row sep=crcr]{%
1	0.9914\\
2	0.9828\\
3	0.7324\\
4	0.0954\\
5	0.0221\\
6	0.0217\\
7	0.0217\\
8	0.0216\\
};
\addlegendentry{$u$ ($J=4$)}

\addplot [color=red, dashed, mark size=2.5pt, mark=square, mark options=solid]
  table[row sep=crcr]{%
1	0.9653\\
2	0.9482\\
3	0.76\\
4	0.2232\\
5	0.046\\
6	0.0096\\
7	0.003\\
8	0.0023\\
};
\addlegendentry{$p$ ($J=4$)}

\addplot [color=blue, mark size=2.5pt, mark=o, mark options=solid]
  table[row sep=crcr]{%
1	0.9909\\
2	0.9798\\
3	0.7729\\
4	0.26\\
5	0.0713\\
6	0.0456\\
7	0.0244\\
8	0.0239\\
};
\addlegendentry{$u$ ($J=2$)}

\addplot [color=blue, dashed, mark size=2.5pt, mark=o, mark options=solid]
  table[row sep=crcr]{%
1	0.9642\\
2	0.9411\\
3	0.6636\\
4	0.1481\\
5	0.0273\\
6	0.0056\\
7	0.0026\\
8	0.0025\\
};
\addlegendentry{$p$ ($J=2$)}

\end{axis}
\end{tikzpicture}%
%
%
\begin{tikzpicture}

\begin{axis}[%
width=2.4in,
height=2in,
scale only axis,
xmin=1,
xmax=8,
xlabel={Number of basis functions $J$},
ymin=0,
ymax=1,
ylabel={Relative energy error},
axis background/.style={fill=white},
legend style={legend cell align=left, align=left, draw=black},
legend style={at={(0.55,0.3)},anchor=west}
]
\addplot [color=red, mark size=2.5pt, mark=square, mark options=solid]
  table[row sep=crcr]{%
1	0.7809\\
2	0.0257\\
3	0.0231\\
4	0.0217\\
5	0.0197\\
6	0.0103\\
7	0.0097\\
8	0.0089\\
};
\addlegendentry{$u$ ($m=6$)}

\addplot [color=red, dashed, mark size=2.5pt, mark=square, mark options=solid]
  table[row sep=crcr]{%
1	0.0092\\
2	0.0056\\
3	0.0094\\
4	0.0096\\
5	0.009\\
6	0.0085\\
7	0.0084\\
8	0.0064\\
};
\addlegendentry{$p$ ($m=6$)}

\addplot [color=blue, mark size=2.5pt, mark=o, mark options=solid]
  table[row sep=crcr]{%
1	0.899\\
2	0.749\\
3	0.7256\\
4	0.7323\\
5	0.6866\\
6	0.6581\\
7	0.655\\
8	0.5591\\
};
\addlegendentry{$u$ ($m=3$)}

\addplot [color=blue, dashed, mark size=2.5pt, mark=o, mark options=solid]
  table[row sep=crcr]{%
1	0.5172\\
2	0.6638\\
3	0.755\\
4	0.76\\
5	0.7199\\
6	0.6867\\
7	0.6847\\
8	0.6162\\
};
\addlegendentry{$p$ ($m=3$)}

\end{axis}
\end{tikzpicture}%
\caption{Relative energy error (test model 1) for $H=\sqrt{2}/40$ and fixed $J$ (left), fixed $m$ (right).}
\label{fig:ovnb1}
\end{figure}

\begin{figure}
	\centering
	\subfigure[Pressure $p$.]{
		\includegraphics[width=2in]{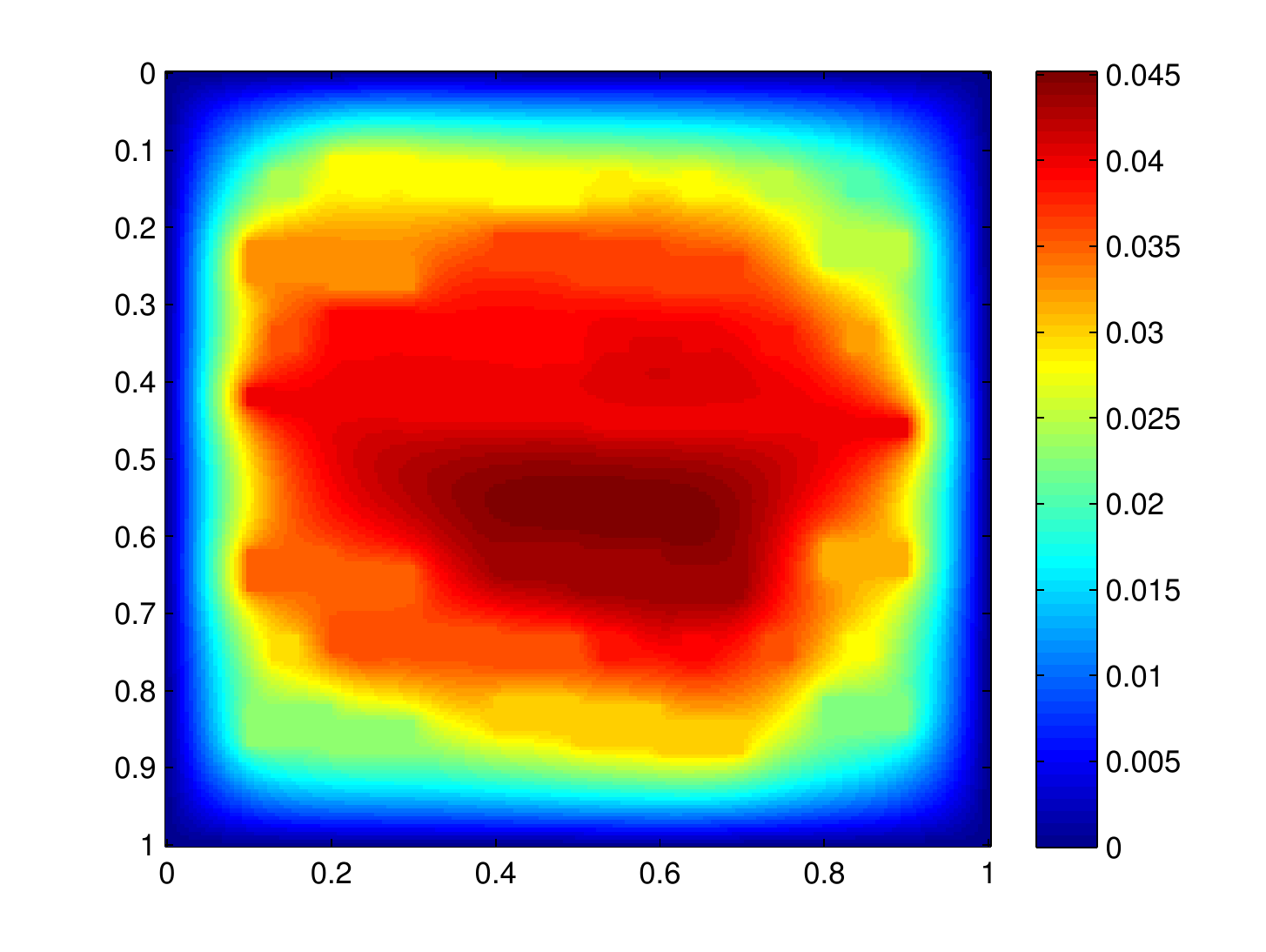}}
	\subfigure[First component of $u$.]{
		\includegraphics[width=2in]{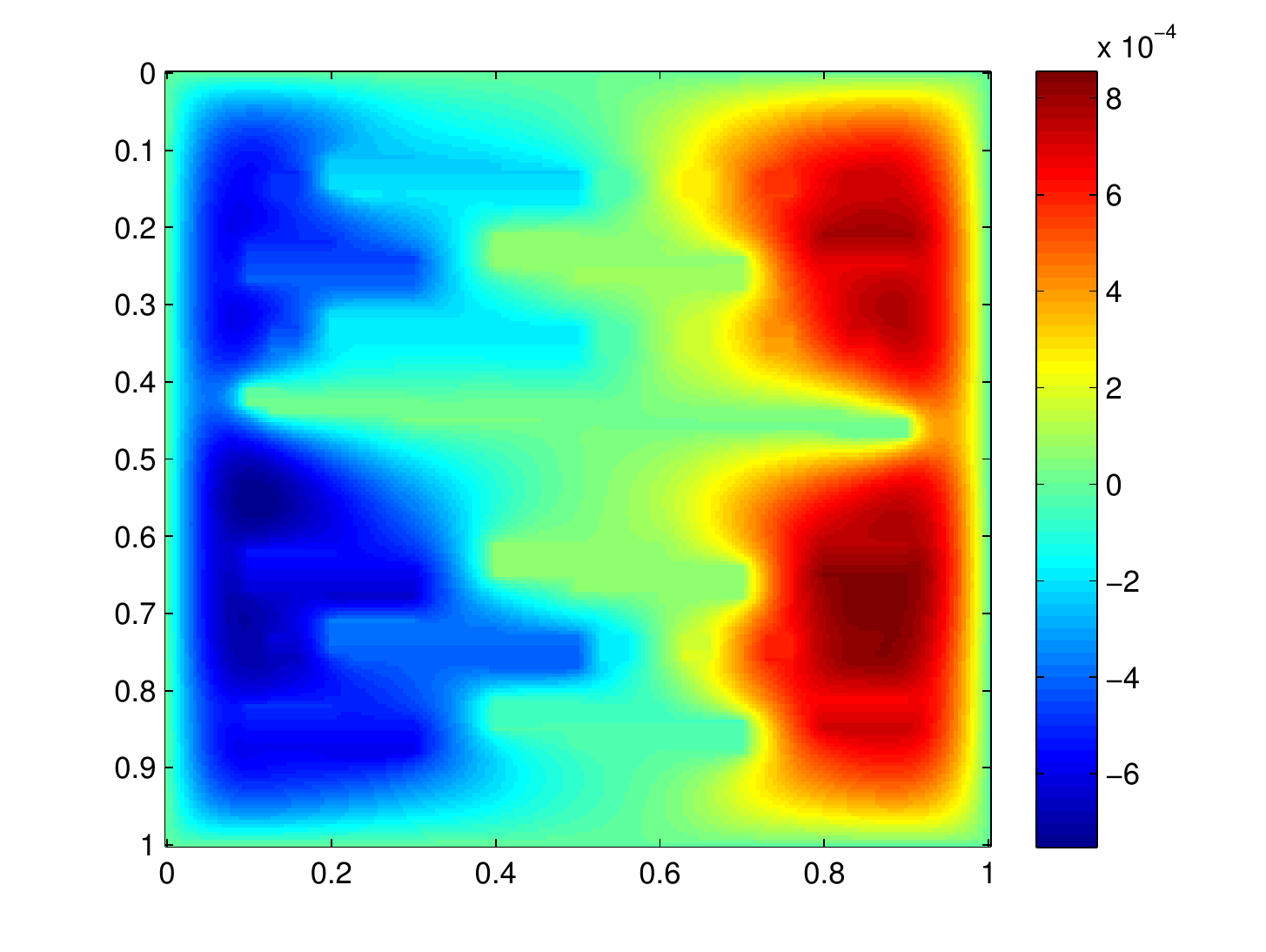}}
	\subfigure[Second component of $u$.]{
		\includegraphics[width=2in]{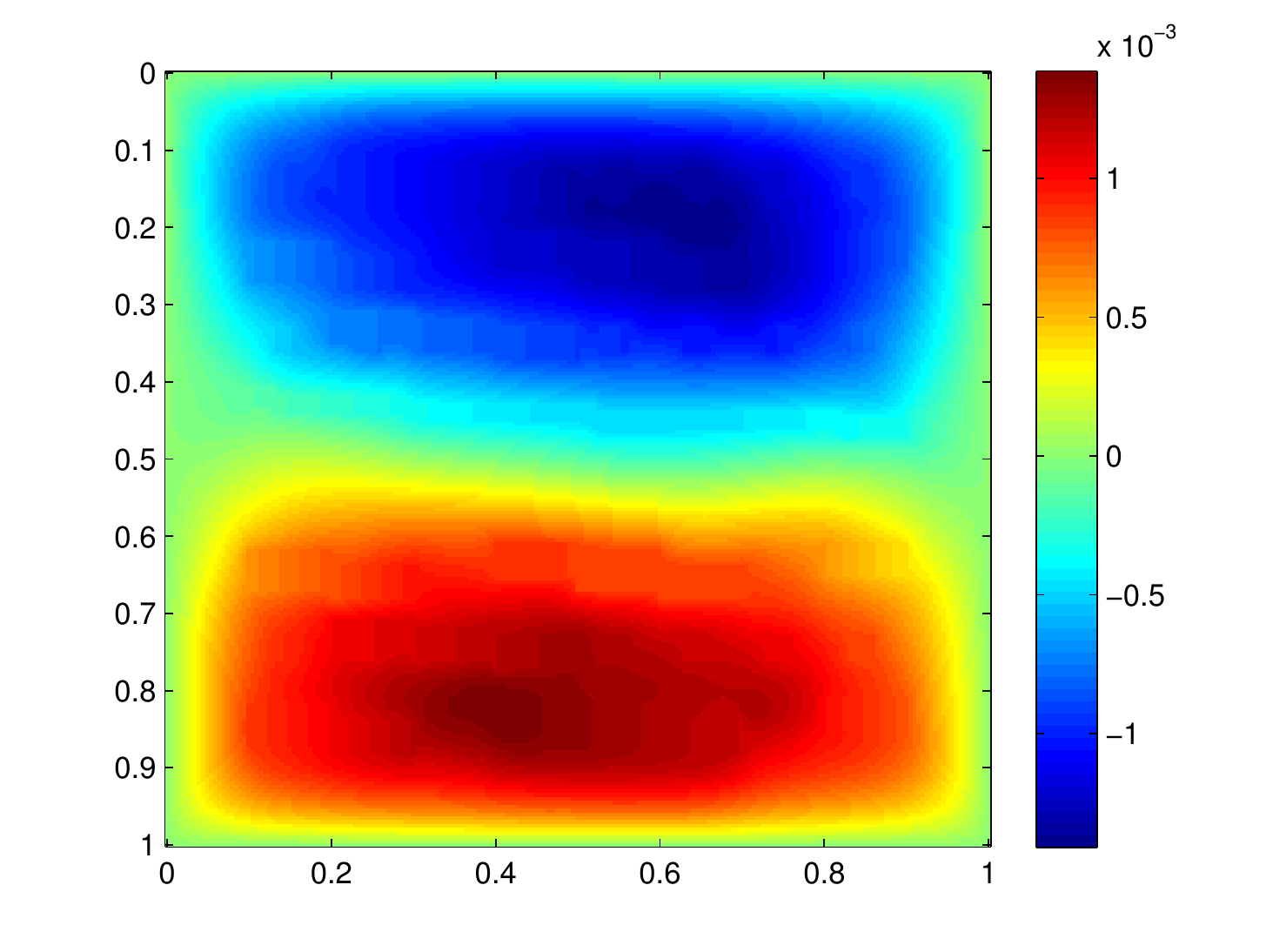}}	
	\caption{Reference solution for test model 1 at $T = 100$.}
	\label{fig:fine1}
\end{figure}

\begin{figure}
	\centering
	\subfigure[Pressure $p$.]{
		\includegraphics[width=2in]{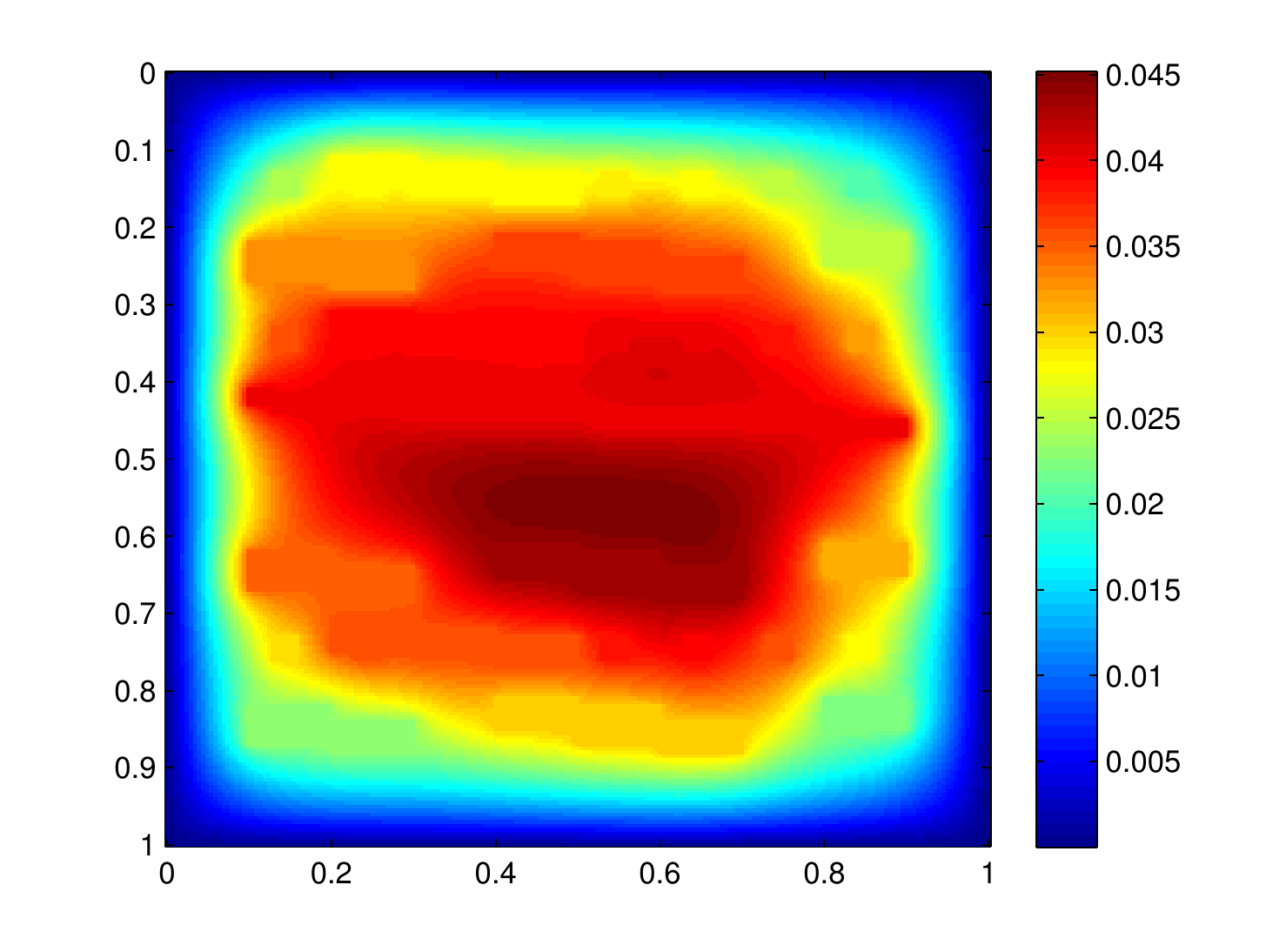}}
	\subfigure[First component of $u$.]{
		\includegraphics[width=2in]{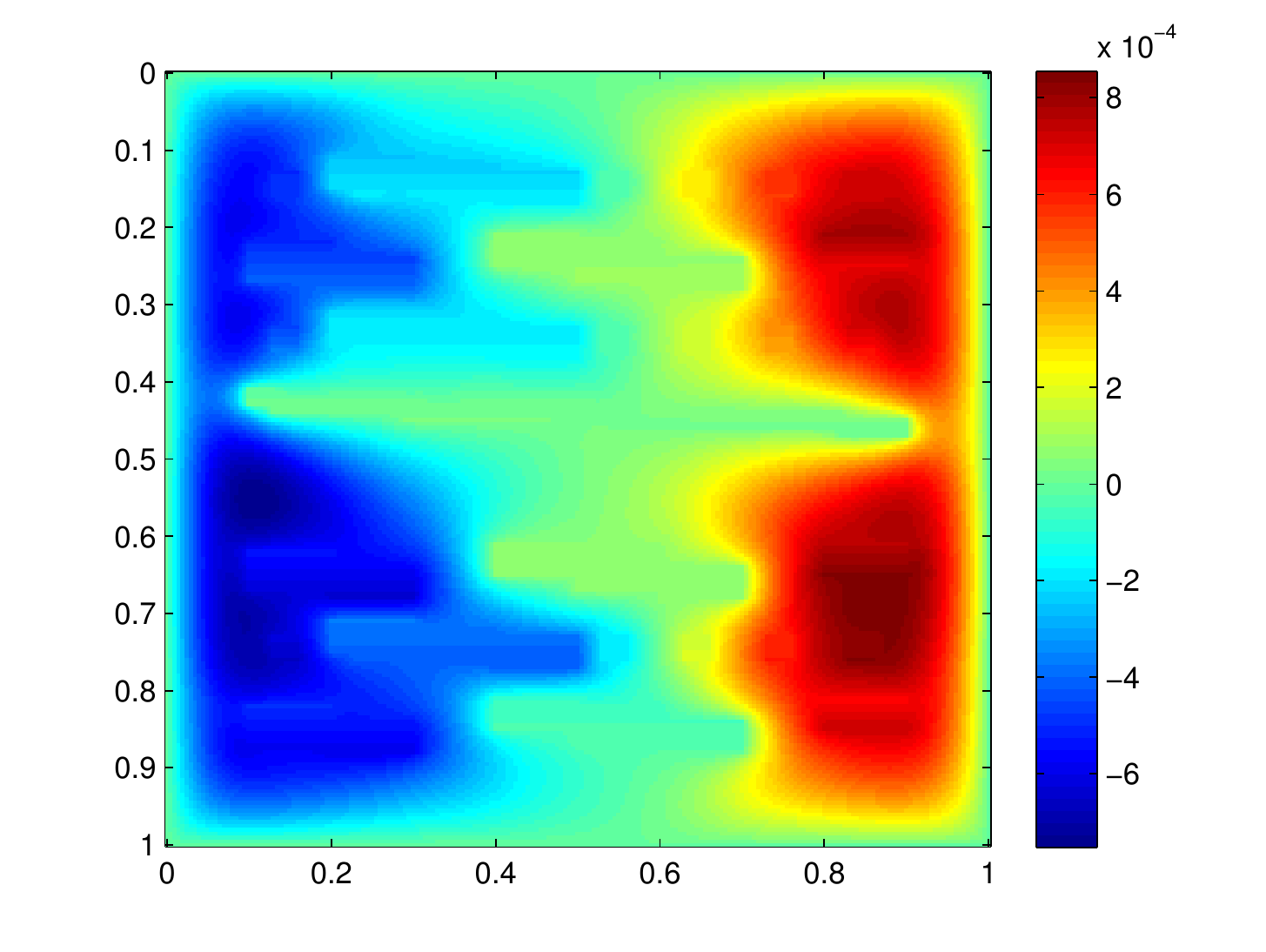}}
	\subfigure[Second component of $u$.]{
		\includegraphics[width=2in]{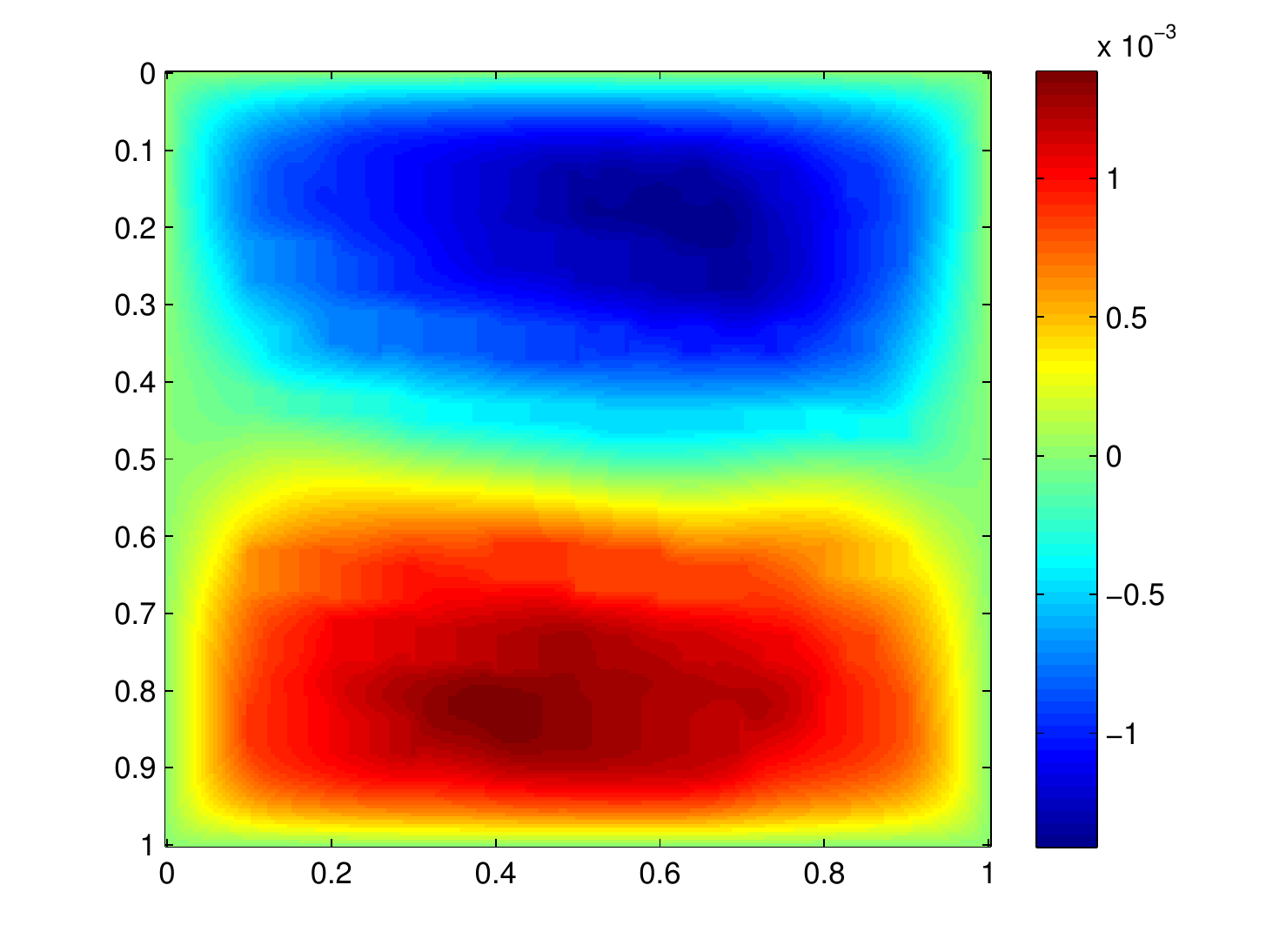}}	
	\caption{Multiscale solution for test model $1$ at $T = 100$ with $H =\sqrt{2}/40$, $m=6$, and $J=4$.}
	\label{fig:ms1} 
\end{figure}

\subsection{Test model 2}
In this subsection, we apply the proposed method to test model $2$ with the same choices for $H$, $J$, and $m$ as in the previous subsection. Table \ref{ta:h2} shows the results for different values of the coarse grid size $H$ (and thus $m$) with a fixed number of local basis functions. The reference and the multiscale solutions are sketched in Figures \ref{fig:fine2} and \ref{fig:ms2}, respectively. As before, the impact of the number of oversampling layers and number of basis functions are investigated and depicted in Figure \ref{fig:ovnb2}. One can observe that the multiscale approximation converges to the reference solution as $H$ decreases. Note that Figure \ref{fig:ovnb2} (left) suggests to choose $m \geq 5$, which is in line with the theoretical findings in \cite{chung2017constraint} that the multiscale features are captured when using a sufficiently large number of oversampling layers. As before, also the number of eigenfunctions in the auxiliary spaces can enhance the performance of the multiscale method, see Figure \ref{fig:ovnb2} (right).
\begin{table}
\centering \begin{tabular}{c|c|c||c|c|c|c}
$J$ &$H$& $m$  & $e_{L^2}^u$   & $e_{a}^u$& $e_{L^2}^p$   & $e_{b}^p$ \tabularnewline\hline
		4&$\sqrt{2}$/10	&4&2.22e-02 &  5.14e-01 & 9.64e-05 & 3.59e-02      \tabularnewline\hline
		4&$\sqrt{2}$/20	&5&3.95e-03 &  2.06e-01 & 2.77e-05 & 1.49e-02      \tabularnewline\hline
		4&$\sqrt{2}$/40	&6&4.94e-04 &  5.60e-02 & 7.81e-06 & 4.50e-03     \tabularnewline
	\end{tabular}
	\caption{Numerical results with varying coarse grid size $H$ for test model 2.}
	\label{ta:h2}
\end{table}

\begin{figure}
	\centering
%
%
\begin{tikzpicture}

\begin{axis}[%
width=2.4in,
height=2in,
scale only axis,
xmin=1,
xmax=8,
xlabel={Number of oversampling layers},
ymin=0,
ymax=1,
ylabel={Relative energy error},
axis background/.style={fill=white},
legend style={legend cell align=left, align=left, draw=black},
legend style={at={(0.55,0.75)},anchor=west}
]
\addplot [color=red, mark size=2.5pt, mark=square, mark options=solid]
  table[row sep=crcr]{%
1	0.9969\\
2	0.9569\\
3	0.3612\\
4	0.0591\\
5	0.0577\\
6	0.0576\\
7	0.0576\\
8	0.0575\\
};
\addlegendentry{$u$ ($J=4$)}

\addplot [color=red, dashed, mark size=2.5pt, mark=square, mark options=solid]
  table[row sep=crcr]{%
1	0.987\\
2	0.9557\\
3	0.5285\\
4	0.0953\\
5	0.0153\\
6	0.0045\\
7	0.0039\\
8	0.0039\\
};
\addlegendentry{$p$ ($J=4$)}

\addplot [color=blue, mark size=2.5pt, mark=o, mark options=solid]
  table[row sep=crcr]{%
1	0.9941\\
2	0.9368\\
3	0.3472\\
4	0.1231\\
5	0.0856\\
6	0.0814\\
7	0.0814\\
8	0.0812\\
};
\addlegendentry{$u$ ($J=2$)}

\addplot [color=blue, dashed, mark size=2.5pt, mark=o, mark options=solid]
  table[row sep=crcr]{%
1	0.9867\\
2	0.9421\\
3	0.4143\\
4	0.0627\\
5	0.012\\
6	0.0085\\
7	0.0085\\
8	0.0084\\
};
\addlegendentry{$p$ ($J=2$)}

\end{axis}
\end{tikzpicture}%
%
%
\begin{tikzpicture}

\begin{axis}[%
width=2.4in,
height=2in,
scale only axis,
xmin=1,
xmax=8,
xlabel={Number of basis functions},
ymin=0,
ymax=1,
ylabel={Relative energy error},
axis background/.style={fill=white},
legend style={legend cell align=left, align=left, draw=black},
legend style={at={(0.55,0.75)},anchor=west}
]
\addplot [color=red, mark size=2.5pt, mark=square, mark options=solid]
  table[row sep=crcr]{%
1	0.651\\
2	0.0821\\
3	0.0751\\
4	0.0576\\
5	0.0434\\
6	0.0361\\
7	0.0336\\
8	0.0296\\
};
\addlegendentry{$u$ ($m=6$)}

\addplot [color=red, dashed, mark size=2.5pt, mark=square, mark options=solid]
  table[row sep=crcr]{%
1	0.0218\\
2	0.0085\\
3	0.0057\\
4	0.0045\\
5	0.0034\\
6	0.0028\\
7	0.0025\\
8	0.0022\\
};
\addlegendentry{$p$ ($m=6$)}

\addplot [color=blue, mark size=2.5pt, mark=o, mark options=solid]
  table[row sep=crcr]{%
1	0.743\\
2	0.3501\\
3	0.342\\
4	0.3624\\
5	0.3464\\
6	0.3261\\
7	0.2963\\
8	0.2588\\
};
\addlegendentry{$u$ ($m=3$)}

\addplot [color=blue, dashed, mark size=2.5pt, mark=o, mark options=solid]
  table[row sep=crcr]{%
1	0.3734\\
2	0.4143\\
3	0.5079\\
4	0.5285\\
5	0.5214\\
6	0.4997\\
7	0.4715\\
8	0.4368\\
};
\addlegendentry{$p$ ($m=3$)}

\end{axis}
\end{tikzpicture}%
	\caption{Relative energy error (test model 2) for $H=\sqrt{2}/40$ and fixed $J$ (left), fixed $m$ (right).}
\label{fig:ovnb2}
\end{figure}
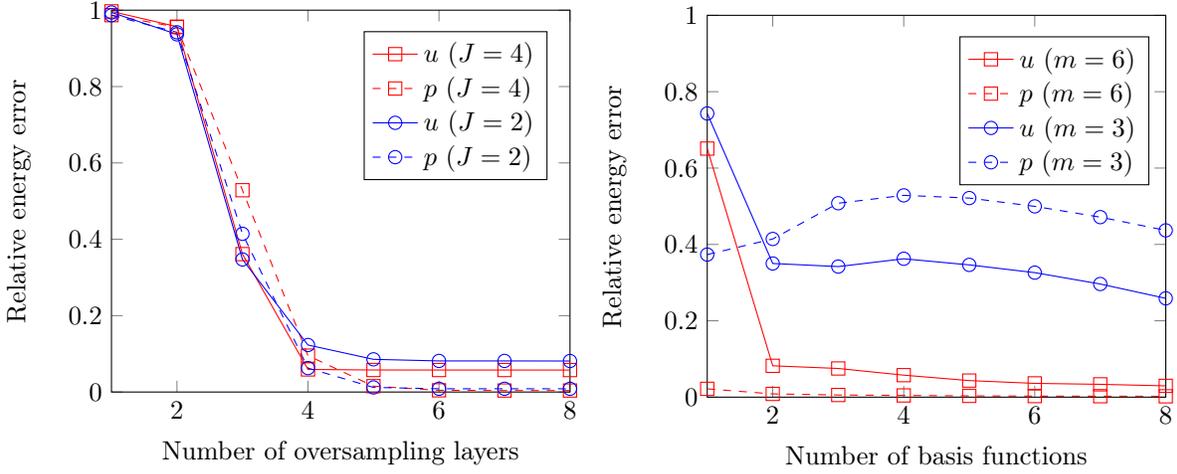

\begin{figure}
	\centering
	\subfigure[Pressure $p$.]{
		\includegraphics[width=2in]{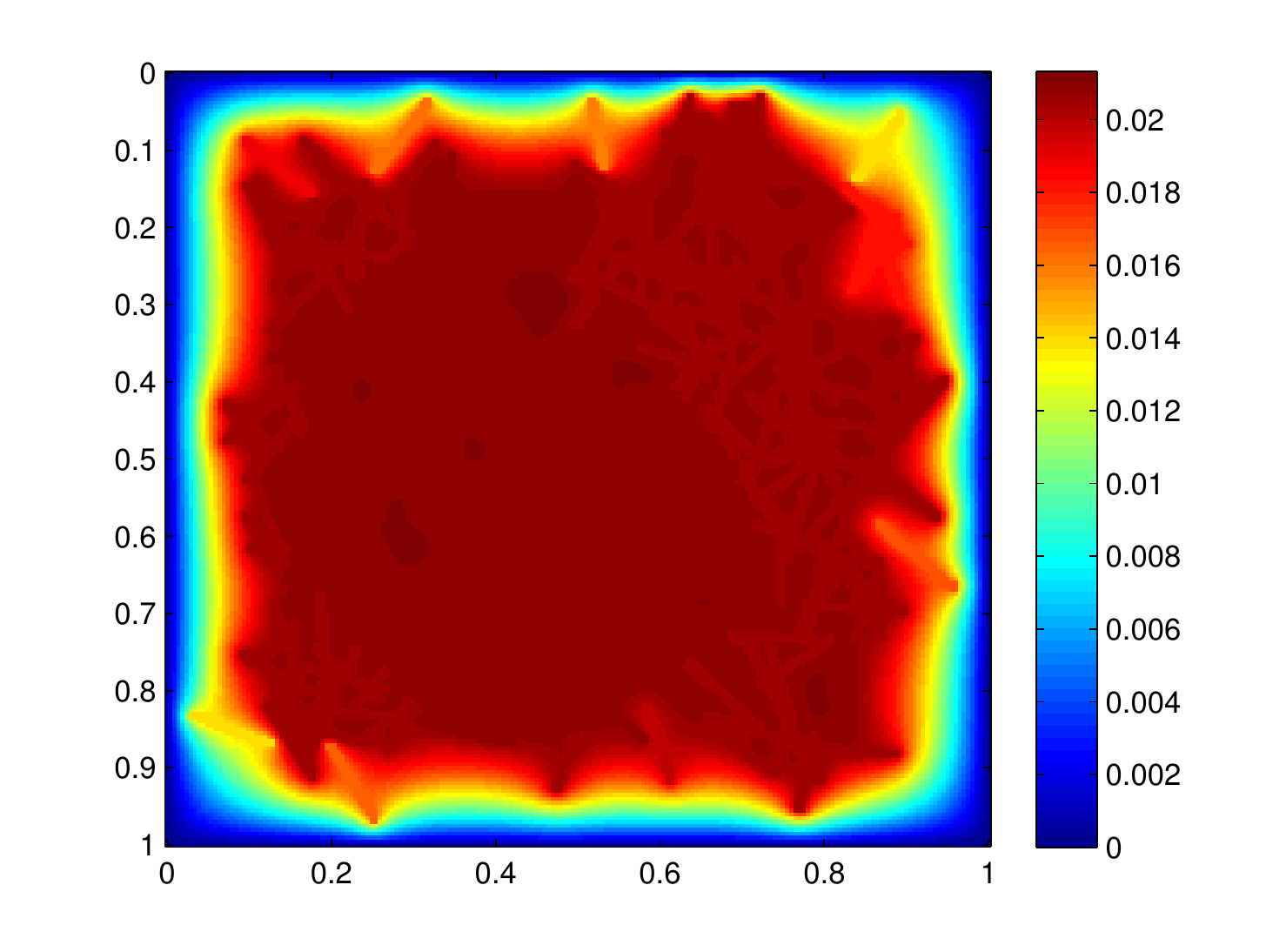}}
	\subfigure[First component of $u$.]{
		\includegraphics[width=2in]{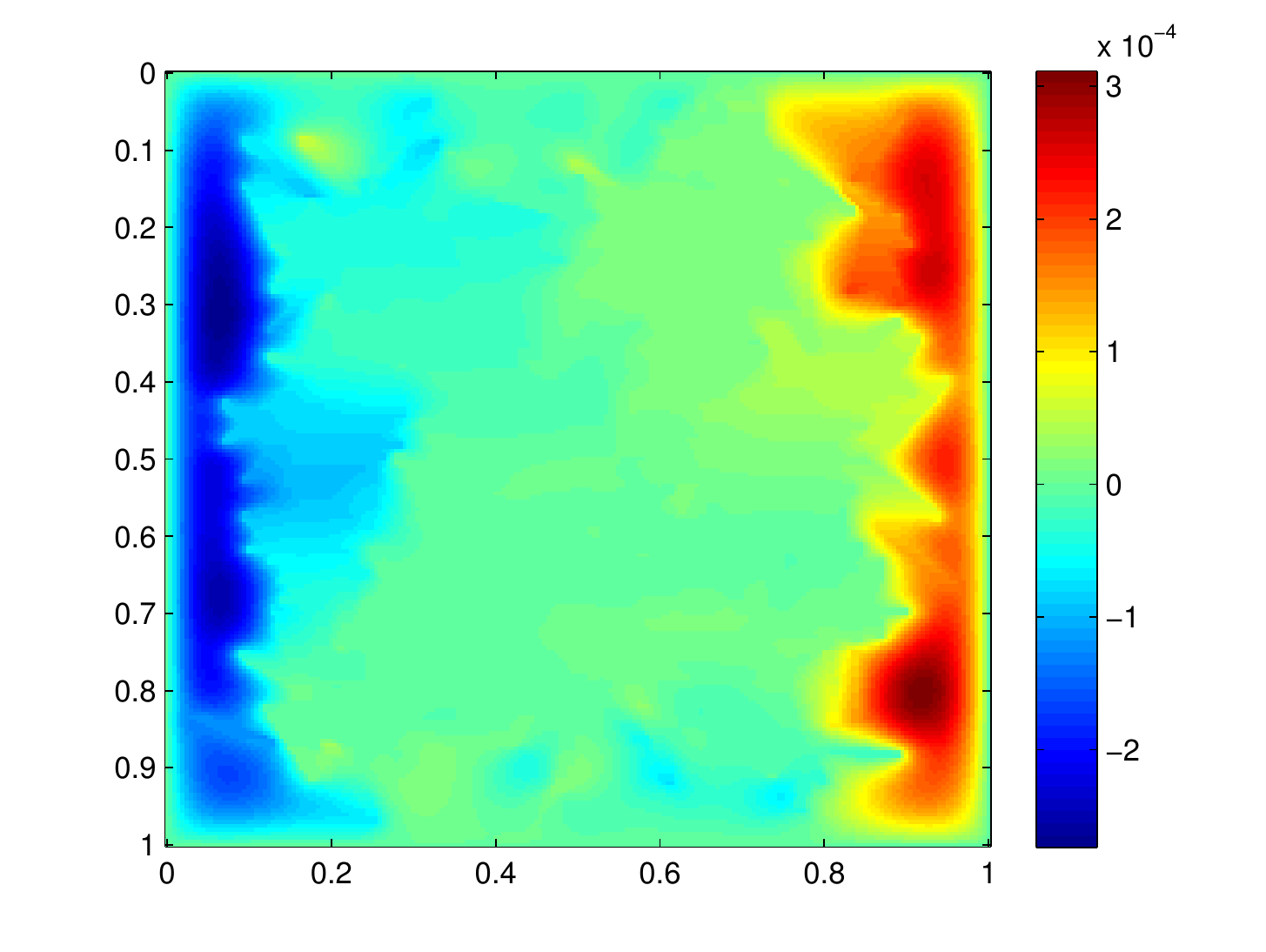}}
	\subfigure[Second component of $u$.]{
		\includegraphics[width=2in]{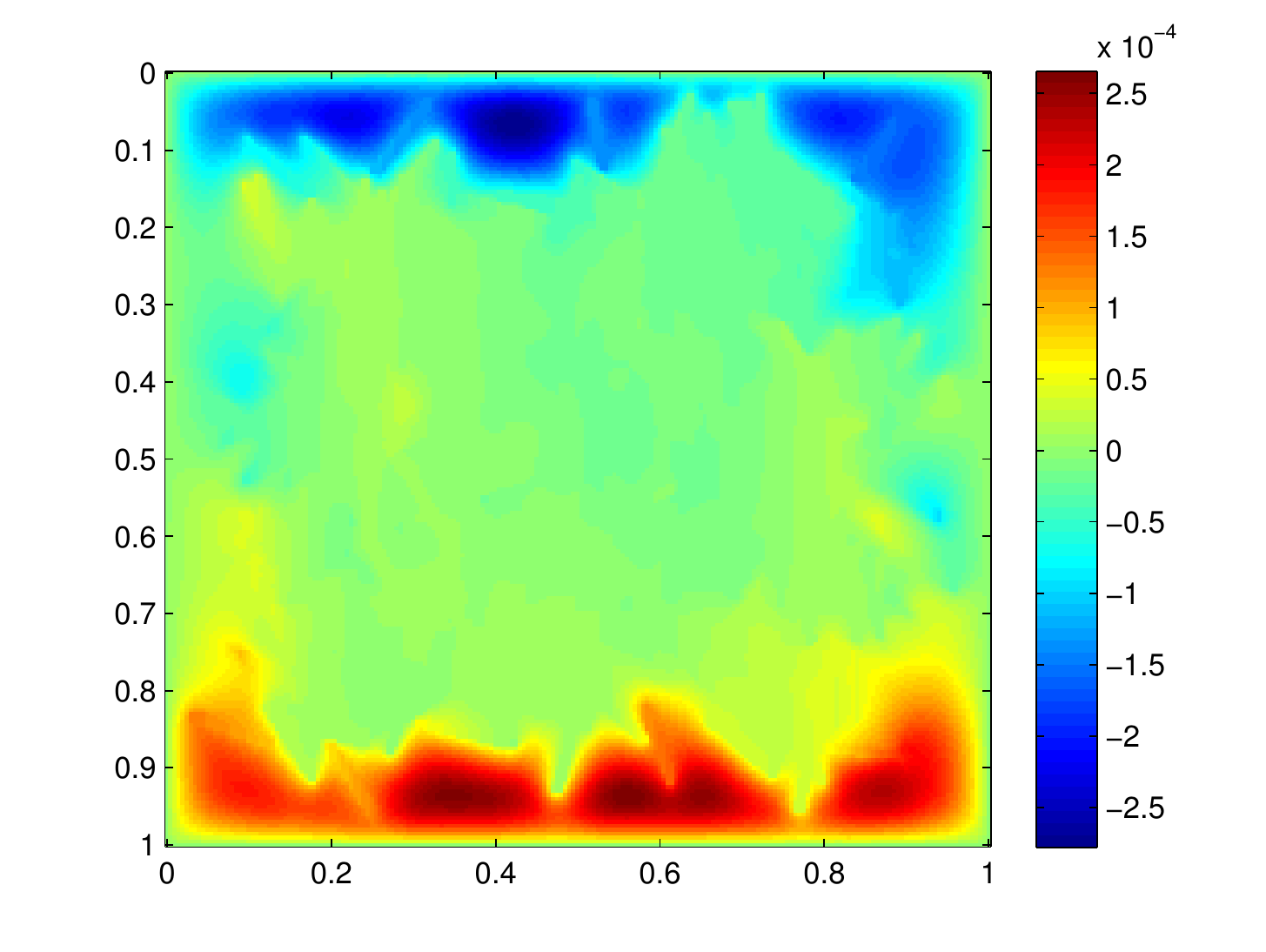}}	
	\caption{Reference solution for test model 2 at $T = 100$.}
	\label{fig:fine2}
\end{figure}

\begin{figure}
	\centering
	\subfigure[Pressure $p$.]{
		\includegraphics[width=2in]{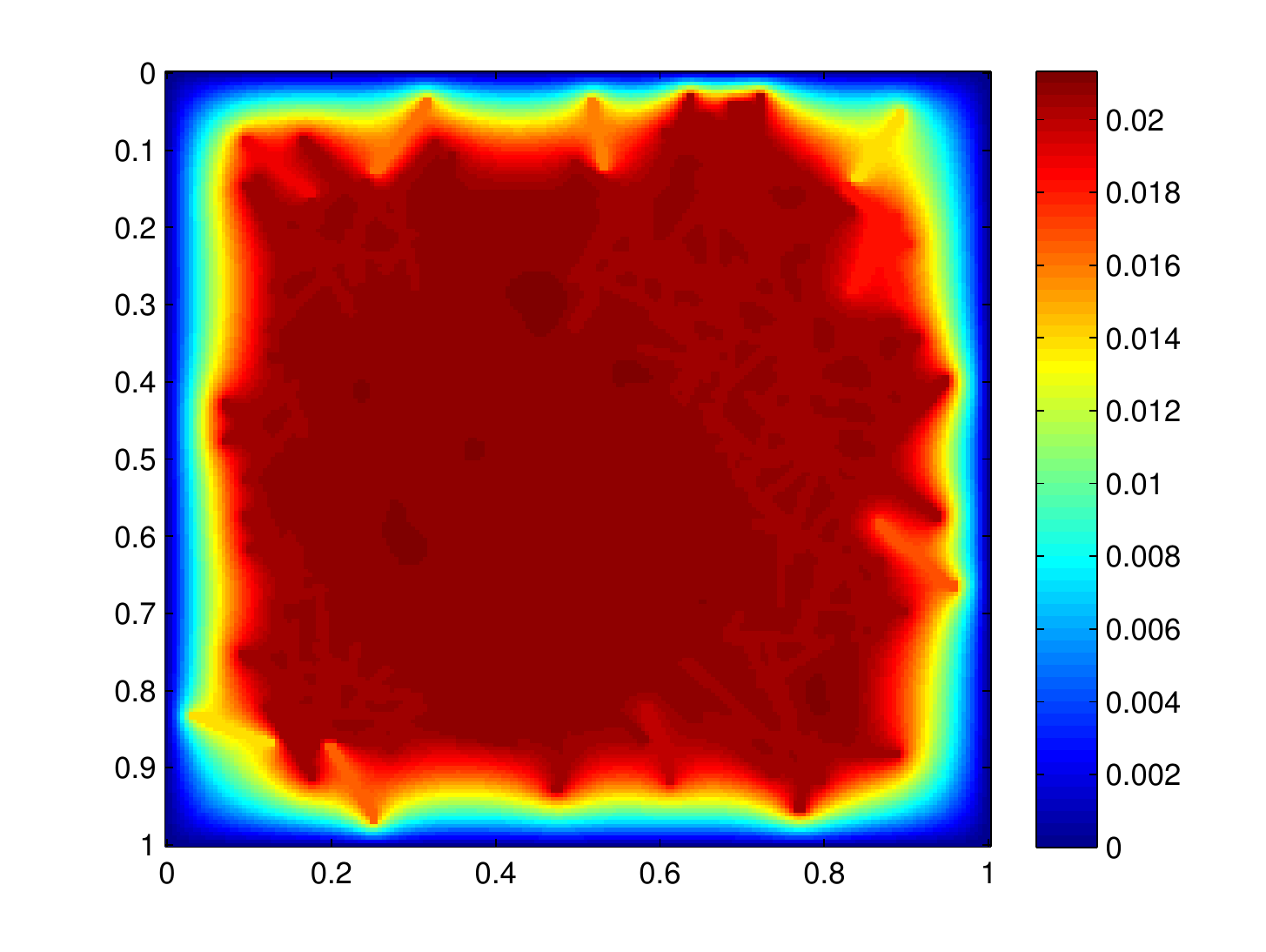}}
	\subfigure[First component of $u$.]{
		\includegraphics[width=2in]{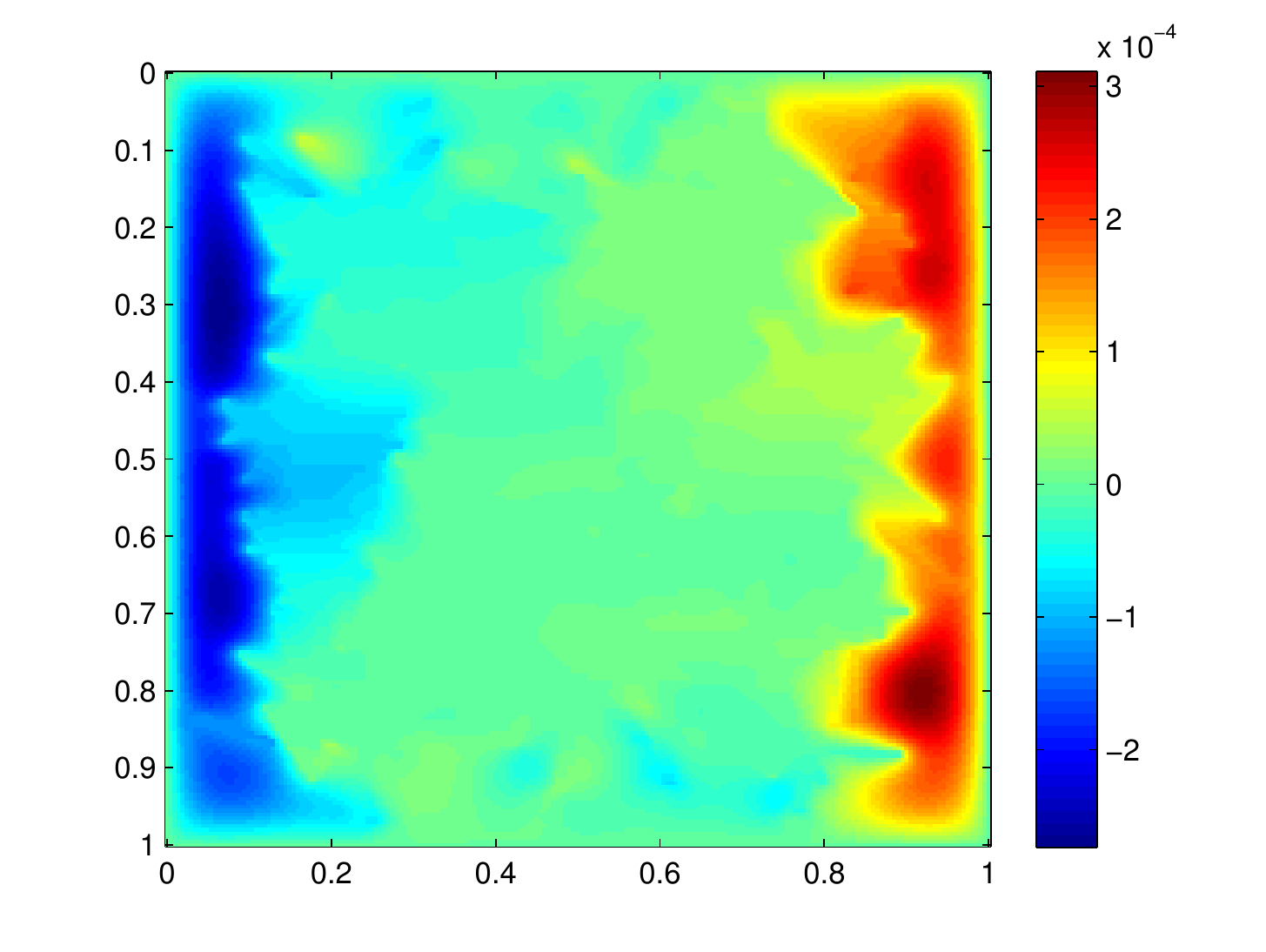}}
	\subfigure[Second component of $u$.]{
		\includegraphics[width=2in]{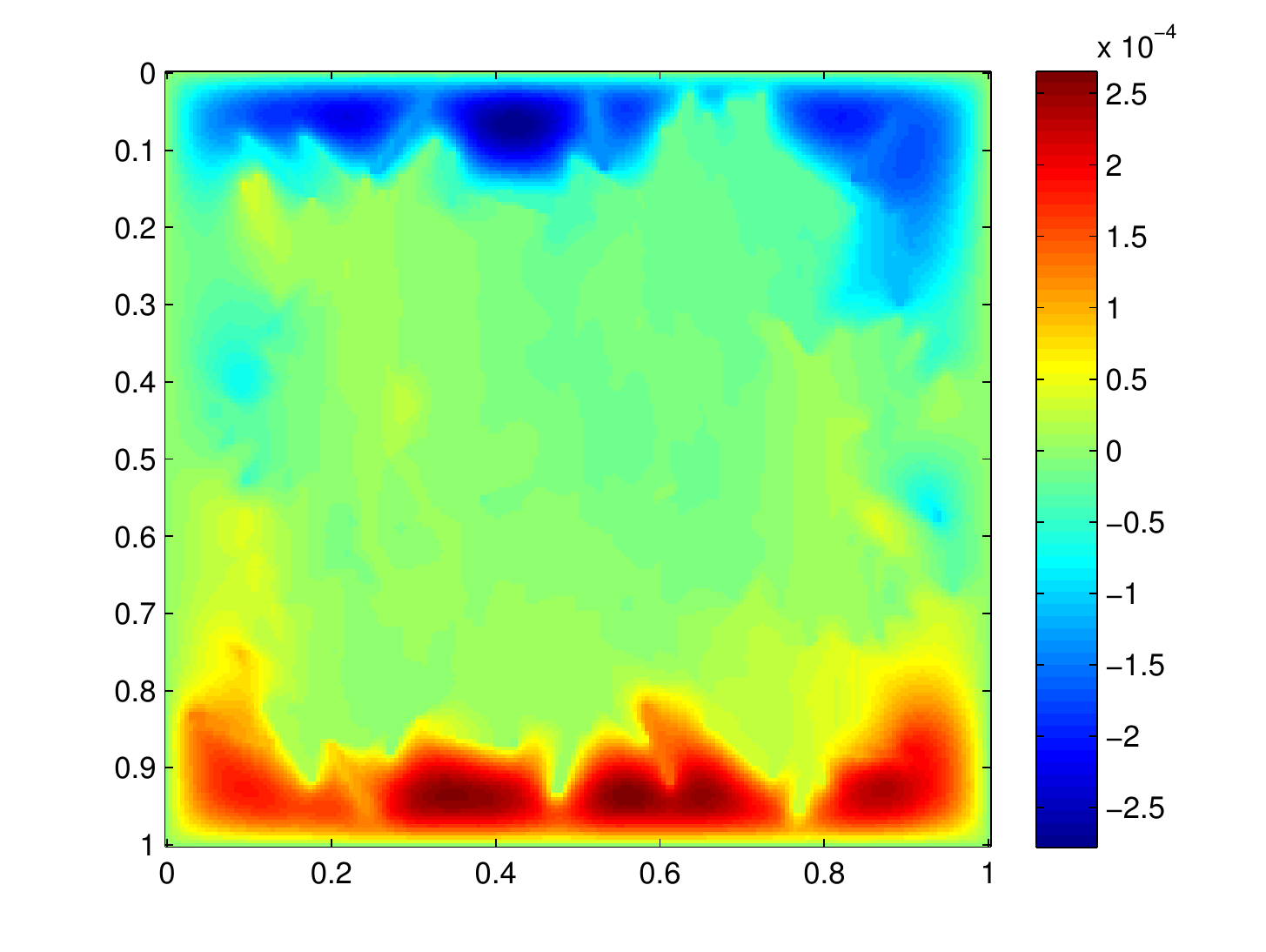}}	
	\caption{Multiscale solution for test model 2 at $T = 100$ with $H =\sqrt{2}/40$, $m=6$, and $J=4$.}
	\label{fig:ms2} 
\end{figure}
\section{Conclusion}\label{sec:con}
In this work, we have proposed a Generalized Multiscale Finite Element Method based on the idea of constraint energy minimization \cite{chung2017constraint} for solving the problem of linear heterogeneous poroelasticity. The spatial discretization is based on CEM-GMsFEM which provides a framework to systematically construct multiscale basis functions for both displacement and pressure. The multiscale basis functions with locally minimal energy are constructed by employing the techniques of oversampling, which leads to an improved accuracy in the simulations. Combined with the implicit Euler scheme for the time discretization, we have shown that the fully discrete method has optimal convergence rates despite the heterogeneities of the media. Numerical results have been presented to illustrate the performance of the proposed method.

\bibliographystyle{plain}
\bibliography{ref_poro}

\begin{thebibliography}{10}

\bibitem{lod_poro}
R.~Altmann, E.~T. Chung, R.~Maier, D.~Peterseim, and S.-M. Pun.
\newblock Computational multiscale methods for linear heterogeneous
  poroelasticity.
\newblock {\em ArXiv preprint 1801.00615}, 2018.

\bibitem{ArPeWY07}
T.~Arbogast, G.~Pencheva, M.~F. Wheeler, and I.~Yotov.
\newblock A multiscale mortar mixed finite element method.
\newblock {\em Multiscale Model. Simul.}, 6(1):319--346, 2007.

\bibitem{bm97}
I.~Babu{\v{s}}ka and J.~M. Melenk.
\newblock The partition of unity method.
\newblock {\em Int. J. Numer. Meth. Engrg.}, 40:727--758, 1997.

\bibitem{biot1941general}
M.~A. Biot.
\newblock General theory of three-dimensional consolidation.
\newblock {\em J. Appl. Phys.}, 12(2):155--164, 1941.

\bibitem{gmsfem_poro1}
D.~L. Brown and M.~Vasilyeva.
\newblock A generalized multiscale finite element method for poroelasticity
  problems {I}: {L}inear problems.
\newblock {\em J. Comput. Appl. Math.}, 294:372--388, 2016.

\bibitem{ch03}
Z.~Chen and T.~Y. Hou.
\newblock A mixed multiscale finite element method for elliptic problems with
  oscillating coefficients.
\newblock {\em Math. Comp.}, 72(242):541--576, 2003.

\bibitem{gmsfem_elasticity}
E.~T. Chung, Y.~Efendiev, and S.~Fu.
\newblock Generalized multiscale finite element method for elasticity
  equations.
\newblock {\em Int. J. Geomath.}, 5(2):225--254, 2014.

\bibitem{chung2017constraint}
E.~T. Chung, Y.~Efendiev, and W.~T. Leung.
\newblock Constraint energy minimizing generalized multiscale finite element
  method.
\newblock {\em Comput. Methods Appl. Mech. Engrg.}, 339:298--319, 2018.

\bibitem{ciarlet1988mathematical}
P.~G. Ciarlet.
\newblock {\em Mathematical elasticity. {V}ol. {I}}.
\newblock North-Holland Publishing Co., Amsterdam, 1988.

\bibitem{durfolsky1991homo}
L.~J. Durlofsky.
\newblock Numerical calculation of equivalent grid block permeability tensors
  for heterogeneous porous media.
\newblock {\em Water Resour. Res.}, 27(5):699--708.

\bibitem{egw10}
Y.~Efendiev, J.~Galvis, and X.-H. Wu.
\newblock Multiscale finite element methods for high-contrast problems using
  local spectral basis functions.
\newblock {\em J. Comput. Phys.}, 230(4):937--955, 2011.

\bibitem{efendiev2009multiscale}
Y.~Efendiev and T.~Y. Hou.
\newblock {\em Multiscale finite element methods: theory and applications},
  volume~4.
\newblock Springer Science \& Business Media, 2009.

\bibitem{EngHMP16}
C.~{Engwer}, P.~{Henning}, A.~{M{\aa}lqvist}, and D.~{Peterseim}.
\newblock Efficient implementation of the localized orthogonal decomposition
  method.
\newblock {\em ArXiv preprint 1602.01658}, 2016.

\bibitem{ern2009posteriori}
A.~Ern and S.~Meunier.
\newblock A posteriori error analysis of euler-galerkin approximations to
  coupled elliptic-parabolic problems.
\newblock {\em M2AN Math. Model. Numer. Anal.}, 43(2):353--375, 2009.

\bibitem{gao2015numerical}
K.~Gao, E.~T. Chung, R.~L. Gibson, S.~Fu, and Y.~Efendiev.
\newblock A numerical homogenization method for heterogeneous, anisotropic
  elastic media based on multiscale theory.
\newblock {\em Geophysics}, 80(4):D385--D401, 2015.

\bibitem{HenP16}
P.~Henning and A.~Persson.
\newblock A multiscale method for linear elasticity reducing poisson locking.
\newblock {\em Comput. Methods Appl. Mech. Engrg.}, 310:156--171, 2016.

\bibitem{HenP13}
P.~Henning and D.~Peterseim.
\newblock Oversampling for the multiscale finite element method.
\newblock {\em Multiscale Model. Simul.}, 11(4):1149--1175, 2013.

\bibitem{hw97}
T.~Y. Hou and X.-H. Wu.
\newblock A multiscale finite element method for elliptic problems in composite
  materials and porous media.
\newblock {\em J. Comput. Phys.}, 134:169--189, 1997.

\bibitem{jennylt03}
P.~Jenny, S.H. Lee, and H.~Tchelepi.
\newblock Multi-scale finite volume method for elliptic problems in subsurface
  flow simulation.
\newblock {\em J. Comput. Phys.}, 187:47--67, 2003.

\bibitem{KorPY18}
R.~Kornhuber, D.~Peterseim, and H.~Yserentant.
\newblock An analysis of a class of variational multiscale methods based on
  subspace decomposition.
\newblock {\em Math. Comp.}, published electronically, 2018.

\bibitem{KorY16}
R.~Kornhuber and H.~Yserentant.
\newblock Numerical homogenization of elliptic multiscale problems by subspace
  decomposition.
\newblock {\em Multiscale Model. Simul.}, 14(3):1017--1036, 2016.

\bibitem{maalqvist2017generalized}
A.~M{\aa}lqvist and A.~Persson.
\newblock A generalized finite element method for linear thermoelasticity.
\newblock {\em ESAIM: M2AN Math. Model. Numer. Anal.}, 51(4):1145--1171, 2017.

\bibitem{maalqvist2014localization}
A.~M{\aa}lqvist and D.~Peterseim.
\newblock Localization of elliptic multiscale problems.
\newblock {\em Math. Comp.}, 83(290):2583--2603, 2014.

\bibitem{mura2016two}
J.~Mura and A.~Caiazzo.
\newblock A two-scale homogenization approach for the estimation of porosity in
  elastic media.
\newblock In {\em Trends in Differential Equations and Applications}, pages
  89--105. Springer, 2016.

\bibitem{Pet16}
D.~Peterseim.
\newblock Variational multiscale stabilization and the exponential decay of
  fine-scale correctors.
\newblock In {\em Building Bridges: Connections and Challenges in Modern
  Approaches to Numerical Partial Differential Equations}, pages 341--367.
  Springer, 2016.

\bibitem{PetS16}
D.~Peterseim and R.~Scheichl.
\newblock Robust numerical upscaling of elliptic multiscale problems at high
  contrast.
\newblock {\em Comput. Methods Appl. Math.}, 16(4):579--603, 2016.

\bibitem{geomechanics}
C.~M. Sayers and P.~M. T.~M. Schutjens.
\newblock An introduction to reservoir geomechanics.
\newblock {\em The Leading Edge}, 26(5):597--601, 2007.

\bibitem{showalter2000diffusion}
R.~E. Showalter.
\newblock Diffusion in poro-elastic media.
\newblock {\em J. Math. Anal. Appl.}, 251(1):310--340, 2000.

\bibitem{wu2002analysis}
X.-H. Wu, Y.~Efendiev, and T.~Y. Hou.
\newblock Analysis of upscaling absolute permeability.
\newblock {\em Discrete Cont. Dyn.-B}, 2(2):185--204, 2002.

\bibitem{zoback2010reservoir}
M.~D. Zoback.
\newblock {\em Reservoir geomechanics}.
\newblock Cambridge University Press, 2010.

\end{thebibliography}

\end{document}